\documentclass[final]{article}

\usepackage{amsmath,amssymb,MnSymbol}
\usepackage{amsthm}
\usepackage{xspace}
\usepackage{enumitem}
\usepackage{booktabs}
\usepackage{tikz}
\usepackage{PCpic}
\usepackage[ruled,vlined,commentsnumbered]{algorithm2e}
\usepackage[section]{placeins}

\usepackage{fixme}
\usepackage[notref,notcite]{showkeys}

\newtheorem{theorem}{{\bf Theorem}}[section]
\newtheorem{corollary}[theorem]{{\bf Corollary}}
\newtheorem{definition}[theorem]{{\bf Definition}}
\newtheorem{example}[theorem]{{\bf Example}}
\newtheorem{lemma}[theorem]{{\bf Lemma}}
\newtheorem{remark}[theorem]{{\bf Remark}}

\newtheorem{proposition}[theorem]{{\bf Proposition}}

\SetKwInOut{Input}{input}
\SetKwInOut{Output}{output}
\DontPrintSemicolon
\LinesNumbered

\newenvironment{enumcases}{\begin{enumerate}[label=\arabic*.)]}{\end{enumerate}}

\newenvironment{enumproperties}{\begin{enumerate}[label=(\roman*)]}{\end{enumerate}}
\newcommand{\refenumproperty}[1]{\ref{#1}}
\newenvironment{enuminvariants}{\begin{enumerate}[label=(\alph*)]}{\end{enumerate}}
\newcommand{\refenuminvariant}[1]{\ref{#1}}

\newcommand{\proc}[1]{\ifmmode\text{\textsc{#1}}\else\textsc{#1}\xspace\fi}

\newcommand{\sign}{\ensuremath{\mathrm{sign}}}
\newcommand{\abs}[1]{\ensuremath{\left\vert #1\right\vert}}

\newcommand{\supp}{\ensuremath{\mathop{\mathrm{supp}}}} 
\newcommand{\cupi}{\ensuremath{\dot\cup}}
\newcommand{\tow}{\ensuremath{\mathop{\mathrm{tow}}}} 

\newcommand{\eps}{\varepsilon}
\renewcommand{\phi}{\varphi}

\newcommand{\Z}{\ensuremath{\mathbb{Z}}}

\newcommand{\Nz}{\ensuremath{{\mathbb{N}_0}}}

\renewcommand{\O}{\ensuremath{\mathcal{O}}}

\newcommand{\sO}{\ensuremath{\tilde{\mathcal{O}}}}

\newcommand{\grpeq}{\sim}
\newcommand{\streq}{=}

\newcommand{\g}[1]{\ensuremath{\langle #1\rangle}}
\newcommand{\gr}[2]{\ensuremath{\langle #1\ \vert\ #2\rangle}}

\newcommand{\e}{\eps} 
\newcommand{\ch}{\ensuremath{\mathrm{ch}}} 
\newcommand{\pot}{\ensuremath{\mathrm{pot}}} 

\newcommand{\clone}{\proc{Clone}}
\newcommand{\PBC}{\proc{Prolong\-Base\-Chain}}
\newcommand{\ERed}{\proc{Extend\-Reduction}}
\newcommand{\Red}{\proc{Reduce}}
\newcommand{\ETree}{\proc{Extend\-Tree}}
\newcommand{\MTree}{\proc{Make\-Tree}}
\newcommand{\CompMark}{\proc{Compactify\-Marking}}
\newcommand{\InsNode}{\proc{Insert\-Node}}
\newcommand{\IncMark}{\proc{Increment\-Marking}}

\newcommand{\BS}{\ensuremath{\mathrm{BS}}}

\newcommand{\BG}[1]{\ensuremath{\mathrm{G}_{(1,#1)}}}
\newcommand{\Hig}[2]{\ensuremath{\mathrm{H}_{#2}(1,#1)}}

\begin{document}
\title{Efficient algorithms for highly compressed data:\\
The Word Problem in\\
Generalized Higman Groups is in P}

\author{J\"urn Laun\\[1ex]
Institut f\"ur Formale Methoden der Informatik\\
Universit\"at Stuttgart, Universit\"atsstra\ss e 38\\
70199 Stuttgart, Germany\\
\texttt{laun@fmi.uni-stuttgart.de}}

\maketitle

\begin{abstract}

\noindent This paper continues the 2012 STACS contribution by Diekert, Ushakov, and the author. We extend the results published in the proceedings in two ways. 

First, we show that the data structure of power circuits can be generalized to work with arbitrary bases $q\ge 2$. This results in a data structure that can hold huge integers, arising by iteratively forming powers of $q$. We show that the properties of power circuits known for $q=2$ translate to the general case. This generalization is non-trivial and additional techniques are required to preserve the time bounds of arithmetic operations that were shown for the case $q=2$. 

The extended power circuit model permits us to conduct operations in the Baumslag-Solitar group $\BS(1,q)$ as efficiently as in $\BS(1,2)$. This allows us to solve the word problem in the generalization $H_4(1,q)$ of Higman's group, which is an amalgamated product of four copies of the Baumslag-Solitar group $\BS(1,q)$ rather than $\BS(1,2)$ in the original form. 

As a second result, we allow arbitrary numbers $f\ge 4$ of copies of $\BS(1,q)$, leading to an even more generalized notion of Higman groups $H_f(1,q)$. We prove that the word problem of the latter can still be solved within the $\O(n^6)$ time bound that was shown for $H_4(1,2)$. 
\end{abstract}



\newpage\section{Introduction}\label{sec:intro}

This work is a sequel to the STACS paper \cite{dlu12efficient} and its journal version \cite{dlu13efficient}. However, we try to keep it self-contained by reiterating everything of importance from the two preceeding papers. We extend their results to a more general version of power circuits. As a consequence, we can apply them to larger classes of groups. 

The group $H_4$ was introduced by Higman in 1951 and served to provide the first known example of a finitely generated infinite simple group \cite{higman51finitely}. It belongs to a family $H_f$ ($f\ge 1$) of groups with $f$ generators and $f$ relators:
\begin{equation*}
	H_f=\gr{a_1,\ldots,a_f}{a_{i+1}a_ia_{i+1}^{-1}=a_i^2\ (i\in\Z/f\Z)}
\end{equation*}

For $f<4$ these groups are trivial, which is easy to see for $f\in\{1,2\}$, but suprisingly hard for $f=3$. The latter case was proven by Hirsch \cite{higman51finitely}, see also \S 23 in \cite{neumann53essay} and \cite{allcock10triangles}. If $f\ge 4$, then $H_f$ is infinite, see \cite{serre02trees}, Section 1.4. Since $H_f$ has no non-trivial normal subgroups of finite index, taking a minimal non-trivial quotient results in a finitely generated infinite simple group. 

Until recently, Higman's Group $H_4$ was a candidate for a ``natural'' group with non-elementary word problem.  This was suggested by the huge compression that this group allows. In fact, there are words of length $n$ over the generators $a_i$ and their inverses, which in the group are equal to $a_1^{\tow_2(n)}$, where $\tow_2(n)$ is the tower function (also called ``tetration'') defined by $\tow_q(0)=1$ and $\tow_q(n+1)=q^{\tow_q(n)}$. However, in \cite{dlu12efficient} it was shown that the word problem of $H_4$ is decidable in $\O(n^6\cdot\log n)$ time and \cite{dlu13efficient} improved this bound to $\O(n^6)$. Both results rely on a data structure called ``power circuit'' which was introduced by Miasnikov, Ushakov, and Won in \cite{muw11pc}. Power circuits had already proven useful in algorithmic group theory. In fact, their invention was entailed by the wish to efficiently solve the word problem in the Baumslag-Gersten group $\BG{2}$ (see \cite{muw11bg} and Section \ref{sec:BG}) which shares with $H_4$ the property of huge compression. (In \cite{muw11bg} the group $\BG{2}$ is called ``Baumslag group''. We use the equally common name ``Baumslag-Gersten group'' to avoid confusion with the Baumslag-Solitar group.) For $q=2$, power circuits have been implemented \cite{muCRAG} and there is a computer program solving the word problem in $\BG{2}$. 

This paper builds on the work of Diekert, Ushakov, and the author conducted in \cite{dlu12efficient} and \cite{dlu13efficient}. Its contributions are twofold: In Section \ref{sec:pc}, we extend power circuits to allow arbitrary bases $q\ge 2$. This necessitates changes in the reduction procedure, the core component of power circuits. In Section \ref{sec:higman}, we generalize Higman's group $H_4$ by replacing the underlying group $\BS(1,2)$ by $\BS(1,q)=\gr{a,t}{tat^{-1}=a^q}$. Power circuits with base $q$ are naturally suited for computations in this group. 

Furthermore, with the help of rewriting systems, we give a constructive method of treating any group $H_f$ ($f\ge 4$) rather that just $H_4$. Combining this with base $q$ power circuits leads to an algorithm for the word problem in generalized Higman groups $H_f(1,q)$ which retains the $\O(n^6)$ time bound proved in \cite{dlu13efficient} for $q=2$ and $f=4$. 

\paragraph{Notation and preliminaries}\label{sec:notation}

Algorithms and decision problems are classified by their time complexity on a random-access machine (RAM). We use the notion of amortized analysis with respect to a potential function, see Section 17.3 in \cite{cormen09introduction}. 

With regard to group theory, we use standard notation and facts that can be found in any textbook on the subject, e.g. \cite{lyndon01combinatorial}. In particular, we apply the technique of Britton reductions for solving the word problem in HNN extensions and amalgamated products. 

Rewriting systems are of particular importance for this work. We assume that the reader is familiar with the basic notions of (local) confluence and termination. See for example the textbook \cite{bo93string} for a quick introduction. 

\section{Power circuits}\label{sec:pc}

Power circuits were introduced by Miasnikov, Ushakov, and Won in \cite{muw11pc}. While the underlying ideas presented in this chapter originate from their work, there are some important differences. First, we use a different (and hopefully more accessible) notation, following \cite{dlu12efficient} and \cite{dlu13efficient}. We also allow multiple markings in one circuit. The most important modification, which also distinguishes this paper from \cite{dlu12efficient} and \cite{dlu13efficient}, is the generalization from base $2$ to arbitrary bases $q\ge 2$. 

\subsection{Power Circuit and Evaluation}
For the rest of this paper, we fix an integer $q\ge 2$ for the base and the interval $D=\{-q+1,\ldots,q-1\}$ (the set of digits). We start with a directed acyclic edge-labelled graph without multi-edges, given by $\Pi=(\Gamma,\delta)$. Here, $\Gamma$ is a finite set which will act as the set of nodes (or vertices). The labelled edges (or arcs) are given by the map $\delta:\Gamma\times\Gamma\to D$ where $\delta(u,v)=0$ means that there is no edge from $u$ to $v$ and $\delta(u,v)=e\neq 0$ implies an edge from $u$ to $v$ labelled with the number $e$. In other words, the edge set is $\supp\delta$, the support of the map $\delta$. In addition, we require that the directed graph $(\Gamma,\supp\delta)$ is acyclic. We shall make this assumption throughout this paper without mentioning it again. For any operation on graphs introduced in this chapter, it will be obvious that acyclicity is preserved. 

A marking\index{power circuit!marking} of $\Pi=(\Gamma,\delta)$ is a mapping $M:\Gamma\to D$. Often we regard $M$ as a labelled subset of the nodes of $\Pi$, where the subset is $\supp M$ and the labels are given by $M\vert_{\supp M}:\supp M\to D$. In this sense, $M=0$ (the constant zero marking) and $M=\emptyset$ (the empty marking) are the same. 

Each node $u\in\Gamma$ induces in a natural way its successor marking\index{successor marking} $\Lambda_u$ defined by
\begin{equation*}
	\Lambda_u:\Gamma\to D;\ v\mapsto M(u,v).
\end{equation*}
Intuitively, the successor marking of a node $u$ consists of the target nodes of edges starting at $u$ and their labels are given by those of the edges. 

The evaluation function\index{power circuit!evaluation} $\e$ assigns a real number to each node and each marking of a graph $\Pi$. As $\Pi$ is acyclic we can give an inductive definition of $\e$:
\begin{align*}
	\e(\emptyset)&=0&\text{where $\emptyset$ is the empty marking,}\cr
	\e(M)&=\sum_{u\in\supp M}M(u)\cdot\e(u)&\text{for all other markings $M$,}\cr
	\e(u)&=q^{\e(\Lambda_u)}&\text{for each node $u\in\Gamma$}
\end{align*}
Note that this implies $\e(u)=1$ for all leaves $u$ (nodes without outgoing edges). For every node $u\in\Gamma$ we have
\begin{equation*}
	\e(\Lambda_u)=\log_q\e(u).
\end{equation*}

\begin{example}
Figure \ref{fig:graph_eval} shows an example of such a graph for $q=3$. The set of nodes is $\Gamma=\{u_1,u_2,u_3,u_4,u_5\}$ and $\delta$ is given by
\begin{equation*}
	\begin{matrix}
		\delta(u_2,u_1)=+1,&\delta(u_3,u_1)=+2,&\delta(u_4,u_1)=-1,&\delta(u_4,u_2)=-2,\cr
		\delta(u_4,u_3)=+1,&\delta(u_5,u_2)=+2,&\delta(u_5,u_3)=+1,&\delta(u_5,u_4)=-2.
	\end{matrix}
\end{equation*}
The nodes evaluate to
\begin{equation*}
	\e(u_1)=1,\ \e(u_2)=3,\ \e(u_3)=9,\ \e(u_4)=9,\text{ and }\e(u_5)=\frac{1}{27}.
\end{equation*}
\end{example}

\begin{figure}[ht]
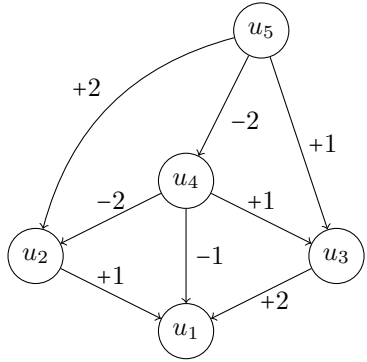

	\begin{center}
		\begin{PCpic}
			\PCnode{u1}{$u_1$}{(0,0)}
			\PCnode{u2}{$u_2$}{(-2,1)}
			\PCnode{u3}{$u_3$}{(2,1)}
			\PCnode{u4}{$u_4$}{(0,2)}
			\PCnode{u5}{$u_5$}{(1,4)}
			\PCedge{u2}{u1}{$+1$}{above}{}
			\PCedge{u3}{u1}{$+2$}{below right=-0.2cm}{}
			\PCedge{u4}{u1}{$-1$}{right}{}
			\PCedge{u4}{u2}{$-2$}{above}{}
			\PCedge{u4}{u3}{$+1$}{above}{}
			\PCedge{u5}{u2}{$+2$}{above left}{bend right}
			\PCedge{u5}{u3}{$+1$}{right}{}
			\PCedge{u5}{u4}{$-2$}{below right=-0.05cm}{}
		\end{PCpic}
	\end{center}
	\caption{Example of an edge-labelled graph ($q=3$)}
	\label{fig:graph_eval}
\end{figure}

\begin{lemma}\label{lem:pc_condition}
Let $\Pi=(\Gamma,\delta)$ be as described above. The following statements are equivalent:
\begin{enumproperties}
\item\label{prop:pc_condition:node} $\e(u)\in q^\Nz=\{q^n:n\in\Nz\}$ for every node $u\in\Gamma$
\item\label{prop:pc_condition:succ} $\e(\Lambda_u)\ge 0$ for every node $u\in\Gamma$
\item\label{prop:pc_condition:mark} $\e(M)\in\Z$ for every marking possible marking $M$ in $\Pi$
\end{enumproperties}
\end{lemma}

\begin{proof}
This is easily seen by noetherian induction with respect to a topological order of $\Gamma$ (i.e., an order compatible with the edges). 
\end{proof}

\begin{definition}
A \emph{power circuit}\index{power circuit} is a finite acyclic edge-labelled graph $\Pi=(\Gamma,\delta)$ without multiple edges that meets the equivalent conditions of Lemma \ref{lem:pc_condition}. 
\end{definition}

\begin{example}
The graph in Figure \ref{fig:graph_eval} is not a power circuit due to $\e(u_5)\not\in\Z$. In contrast, Figure \ref{fig:pc_mark_eval} depicts a power circuit for $q=2$. The values of the nodes are given for illustrative purposes only. In general, these number become too large to be computed. The marking $M$ evaluates to $\e(M)=29$. 
\end{example}

\begin{figure}[ht]
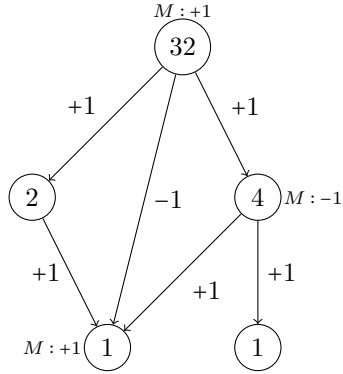

	\begin{center}
		\begin{PCpic}
			\PCnode{1}{$1$}{(0,0)}
			\PCnode{1a}{$1$}{(2,0)}
			\PCnode{2}{$2$}{(-1,2)}
			\PCnode{4}{$4$}{(2,2)}
			\PCnode{32}{$32$}{(1,4)}
			\PCedge{2}{1}{$+1$}{left}{}
			\PCedge{4}{1}{$+1$}{below right}{}
			\PCedge{4}{1a}{$+1$}{right}{}
			\PCedge{32}{1}{$-1$}{right}{}
			\PCedge{32}{2}{$+1$}{above left}{}
			\PCedge{32}{4}{$+1$}{above right}{}
			\PCmark{32}{M}{+1}{above}
			\PCmark{4}{M}{-1}{right}
			\PCmark{1}{M}{+1}{left}
		\end{PCpic}
	\end{center}
	\caption{Example of a power circuit ($q=2$)}
	\label{fig:pc_mark_eval}
\end{figure}

In Corollary \ref{cor:test_graph_pc} we will show that it can be efficiently tested whether a given graph is a power circuit. In \cite{muw11pc}, a power circuit does not have to satisfy the criteria of Lemma \ref{lem:pc_condition}, but if it does, it is called proper. In this sense, we only deal with proper power circuits. 

Figure \ref{fig:pc_tower} shows that a power circuit of linear size can contain markings with values the magnitude of the tower function. 

\begin{figure}[ht]
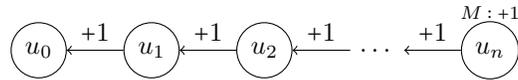

	\begin{center}
		\begin{PCpic}
			\PCnode{1}{$u_0$}{(0,0)}
			\PCnode{2}{$u_1$}{(1.5,0)}
			\PCnode{3}{$u_2$}{(3.0,0)}
			\node (4) at (4.5,0) {\ldots};
			\PCnode{5}{$u_n$}{(6.0,0)}
			\PCedge{2}{1}{$+1$}{above}{}
			\PCedge{3}{2}{$+1$}{above}{}
			\PCedge{4}{3}{$+1$}{above}{}
			\PCedge{5}{4}{$+1$}{above}{}
			\PCmark{5}{M}{+1}{above}
		\end{PCpic}\hspace{5mm}
	\end{center}
	\caption[Marking with tower function value]{Marking with value $\e(M)=\tow_q(n)$}
	\label{fig:pc_tower}
\end{figure}

\subsection{Arithmetic Operations}\label{sec:pc_op}

Let $\Pi=(\Gamma,\delta)$ be a power circuit and $u\in\Gamma$ a node. The operation \clone with result $v=\clone(u)$ creates a new node $v$ with the same successor marking as $u$, but no incoming arcs. We extend this operation to markings $M$, by cloning every single node in $\supp M$. The resulting marking $\clone(M)$ is defined as the marking consisting of all these clones, and the signs are copied from $M$:
\begin{align*}
	\clone(M):\Gamma\cup\{\clone(u)\>:\>u\in\supp M\}\to D;\ &\clone(u)\mapsto M(u),\cr
	&\Gamma\ni u\mapsto 0.
\end{align*}

\begin{example}
In Figure \ref{fig:pc_clone}, the marking $M$ consisting of two nodes is cloned. 
\end{example}

\begin{figure}[ht]
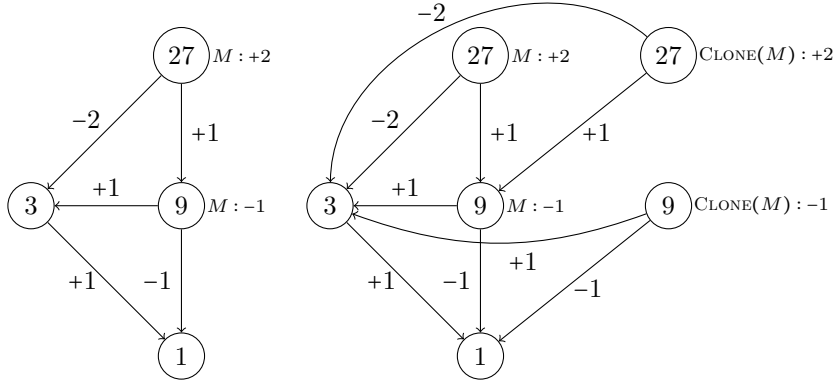

	\begin{center}
		\begin{PCpic}
			\PCnode{1}{$1$}{(0,0)}
			\PCnode{3}{$3$}{(-2,2)}
			\PCnode{9}{$9$}{(0,2)}
			\PCnode{27}{$27$}{(0,4)}
			\PCedge{3}{1}{$+1$}{left}{}
			\PCedge{9}{1}{$-1$}{left}{}
			\PCedge{9}{3}{$+1$}{above}{}
			\PCedge{27}{3}{$-2$}{above left=-0.1}{}
			\PCedge{27}{9}{$+1$}{right}{}
			\PCmark{27}{M}{+2}{right}
			\PCmark{9}{M}{-1}{right}
		\end{PCpic}
		\begin{PCpic}
			\PCnode{1}{$1$}{(0,0)}
			\PCnode{3}{$3$}{(-2,2)}
			\PCnode{9}{$9$}{(0,2)}
			\PCnode{27}{$27$}{(0,4)}
			\PCnode{9a}{$9$}{(2.5,2)}
			\PCnode{27a}{$27$}{(2.5,4)}
			\PCedge{3}{1}{$+1$}{left}{}
			\PCedge{9}{1}{$-1$}{left}{}
			\PCedge{9}{3}{$+1$}{above}{}
			\PCedge{27}{3}{$-2$}{above left=-0.1}{}
			\PCedge{27}{9}{$+1$}{right}{}
			\PCedge{9a}{1}{$-1$}{below right=-0.2cm}{}
			\PCedge{9a}{3}{$+1$}{below right}{bend left=20}
			\PCedge{27a}{3}{$-2$}{above left=-0.1}{bend right=65}
			\PCedge{27a}{9}{$+1$}{right}{}
			\PCmark{27}{M}{+2}{right}
			\PCmark{9}{M}{-1}{right}
			\PCmark{27a}{\clone(M)}{+2}{right}
			\PCmark{9a}{\clone(M)}{-1}{right}
		\end{PCpic}
	\end{center}
	\caption{Cloning a marking ($q=3$)}
	\label{fig:pc_clone}
\end{figure}

Now we can define arithmetic operations. Let $\Pi=(\Gamma,\delta)$ be a power circuit and let $K$ and $M$ be markings in $\Pi$. If the supports of $K$ and $M$ are disjoint, the mapping $K+M$ defined by $(K+M)(u)=K(u)+M(u)$ is a marking with $\e(K+M)=\e(K)+\e(M)$. In general, however, the operands $K$ and $M$ will not be disjoint. In this case we have nodes $u\in\supp K\cap\supp M$ with $K(u)+M(u)\not\in D$, hence $K+M$ is not a valid marking. We solve this problem by cloning: for every node $u$ with $K(u)+M(u)\not\in D$, we create a clone $u^\prime=\clone(u)$ and modify $K+M$ by putting $(K+M)(u):=K(u)$ and $(K+M)(u^\prime):=M(u)$. We obtain a valid marking in the (now enlarged) circuit with value $\e(K)+\e(M)$. 

\begin{example}
In Figure \ref{fig:pc_add}, $\e(K)=7$ and $\e(M)=35$ are added. In the resulting marking, the node with value $1$ cancels out, whereas both the original node with value $4$ and its newly created clone are included. 
\end{example}

\begin{figure}[ht]
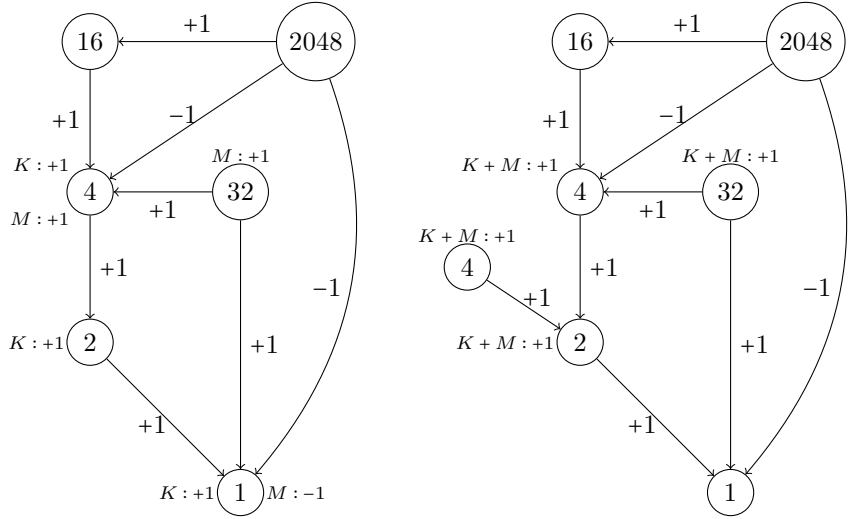

	\begin{center}
		\begin{PCpic}
			\PCnode{1}{$1$}{(0,0)}
			\PCnode{2}{$2$}{(-2,2)}
			\PCnode{4}{$4$}{(-2,4)}
			\PCnode{16}{$16$}{(-2,6)}
			\PCnode{32}{$32$}{(0,4)}
			\PCnode{2048}{$2048$}{(1,6)}
			\PCedge{2}{1}{$+1$}{below left=-0.2cm}{}
			\PCedge{4}{2}{$+1$}{right}{}
			\PCedge{16}{4}{$+1$}{left}{}
			\PCedge{32}{1}{$+1$}{right}{}
			\PCedge{32}{4}{$+1$}{below}{}
			\PCedge{2048}{1}{$-1$}{left}{bend left}
			\PCedge{2048}{4}{$-1$}{above left=-0.2cm}{}
			\PCedge{2048}{16}{$+1$}{above}{}
			\PCmark{4}{K}{+1}{above left}
			\PCmark{2}{K}{+1}{left}
			\PCmark{1}{K}{+1}{left}
			\PCmark{32}{M}{+1}{above}
			\PCmark{4}{M}{+1}{below left}
			\PCmark{1}{M}{-1}{right}
		\end{PCpic}
		\begin{PCpic}
			\PCnode{1}{$1$}{(0,0)}
			\PCnode{2}{$2$}{(-2,2)}
			\PCnode{4}{$4$}{(-2,4)}
			\PCnode{4a}{$4$}{(-3.5,3)}
			\PCnode{16}{$16$}{(-2,6)}
			\PCnode{32}{$32$}{(0,4)}
			\PCnode{2048}{$2048$}{(1,6)}
			\PCedge{2}{1}{$+1$}{below left=-0.2cm}{}
			\PCedge{4}{2}{$+1$}{right}{}
			\PCedge{4a}{2}{$+1$}{above right=-0.2}{}
			\PCedge{16}{4}{$+1$}{left}{}
			\PCedge{32}{1}{$+1$}{right}{}
			\PCedge{32}{4}{$+1$}{below}{}
			\PCedge{2048}{1}{$-1$}{left}{bend left}
			\PCedge{2048}{4}{$-1$}{above left=-0.2cm}{}
			\PCedge{2048}{16}{$+1$}{above}{}
			\PCmark{32}{K+M}{+1}{above}
			\PCmark{4}{K+M}{+1}{above left}
			\PCmark{4a}{K+M}{+1}{above}
			\PCmark{2}{K+M}{+1}{left}
		\end{PCpic}
	\end{center}
	\caption{Addition of markings ($q=2$)}
	\label{fig:pc_add}
\end{figure}

The second operation that we need is multiplication by a power of $q$. We observe that
\begin{equation*}
	\e(K)\cdot q^{\e(M)}
	=\sum_{u\in\supp K}q^{\e(\Lambda_u)}\cdot q^{\e(M)}
	=\sum_{u\in\supp K}q^{\e(\Lambda_u)+\e(M)},
\end{equation*}
so in principle we would just have to introduce new edges from each node in $\supp K$ to each node in $\supp M$. The label of such an edge would be the value that $M$ assigns to the respective target node. This operation works as long as
\begin{enumerate}
\item no cycles are introduced into the circuit,
\item no multi-edges between two nodes are introduced,
\item there are no edges between nodes in $\supp K$, and
\item no other marking in the circuit is affected. (Note here, that the original value of $K$ is lost in any case.)
\end{enumerate}

Again, the solution is cloning. Create clones $K^\prime:=\clone(K)$ and $M^\prime:=\clone(M)$ and introduce new edges by putting $\delta(u,v):=M(v)$ for all $u\in\supp K^\prime$, $v\in\supp M^\prime$. Being clones, nodes in $\supp K^\prime$ and $\supp M^\prime$ have no incoming edges, which prevents cycles and multi-edges. Also, no other marking in the circuit depends on $K^\prime$ or $M^\prime$ directly (by containing these nodes) or indirectly (by containing nodes that are topologically above any node in $K^\prime$ or $M^\prime$). An example (in which no further cloning is necessary) is shown in Figure \ref{fig:pc_mult_power}. 

\begin{figure}[ht]
	\begin{center}
		\begin{PCpic}
			\PCnode{1}{$1$}{(0,0)}
			\PCnode{2}{$2$}{(-1,2)}
			\PCnode{4}{$4$}{(1,4)}
			\PCnode{1a}{$1$}{(3,0)}
			\PCnode{2a}{$2$}{(1,2)}
			\PCnode{4a}{$4$}{(3,2)}
			\PCedge{2}{1}{$+1$}{below left=-0.2cm}{}
			\PCedge{2a}{1}{$+1$}{above left=-0.1cm}{}
			\PCedge{4}{2}{$+1$}{left}{}
			\PCedge{4a}{1}{$+1$}{above left=-0.15cm}{}
			\PCedge{4a}{1a}{$+1$}{right}{}
			\PCmark{4}{K}{+1}{above}
			\PCmark{2a}{K}{+1}{above}
			\PCmark{4a}{M}{+1}{above}
			\PCmark{1a}{M}{+1}{below}
		\end{PCpic}\hspace{5mm}
		\begin{PCpic}
			\PCnode{1}{$1$}{(0,0)}
			\PCnode{2}{$2$}{(-1,2)}
			\PCnode{4}{$4$}{(1,4)}
			\PCnode{1a}{$1$}{(3,0)}
			\PCnode{2a}{$2$}{(1,2)}
			\PCnode{4a}{$4$}{(3,2)}
			\PCedge{2}{1}{$+1$}{below left=-0.2cm}{}
			\PCedge{2a}{1}{$+1$}{above left=-0.1cm}{}
			\PCedge{4}{2}{$+1$}{left}{}
			\PCedge{4a}{1}{$+1$}{above left=-0.15cm}{}
			\PCedge{4a}{1a}{$+1$}{right}{}
			\PCedge{2a}{4a}{$+1$}{above}{}
			\PCedge{2a}{1a}{$+1$}{above right=-0.2cm}{bend left=15}
			\PCedge{4}{4a}{$+1$}{above right=-0.2cm}{}
			\PCedge{4}{1a}{$+1$}{above right=-0.2cm}{bend left=70}
			\PCmark{4}{K\cdot q^M}{+1}{above}
			\PCmark{2a}{K\cdot q^M}{+1}{above}
		\end{PCpic}
	\end{center}
	\caption[Multiplication of $\e(K)$ by $q^{\e(M)}$]{Multiplication of $\e(K)$ by $q^{\e(M)}$ ($q=2$)}
	\label{fig:pc_mult_power}
\end{figure}

Finally, note that the operation $M\mapsto -M$ which negates the value of $M$ is easy to conduct without any complications or the need for cloning. 

\FloatBarrier
\subsection{Reduction}

The operations $K+M$ and $K\cdot q^M$ introduced in the previous section are quite efficient. Assuming that the graph is stored using adjacency lists, the time they take depends only on the size of the markings $M$ and $K$, not on the size of the circuit. The price for this efficieny is that the structure of a power circuit can quickly become rather intransparent. In particular, it is unclear how (in)equality of the values of two markings can be determined in an arbitrary circuit. Again, note that evaluating the nodes or markings is not an option, due to the vast growth permitted by power circuits. For this reason, we restrict ourselves to a subclass of circuits and augment them with some additional data:

\begin{definition}\label{def:reduced_pc}
A \emph{reduced power circuit}\index{power circuit!reduced} is a power circuit $\Pi=(\Gamma,\delta)$ together with a list $(u_1,\ldots,u_n)$ of its nodes and a bit vector $(b_1,\ldots,b_{n-1})\in\mathbb{B}^{n-1}$ such that
\begin{enumproperties}
\item different nodes evaluate to different numbers, i.e., for all $u,v\in\Gamma$ with $u\neq v$, $\e(u)\neq\e(v)$,
\item the list of nodes is sorted by value, i.e., $\e(u_1)<\e(u_2)<\ldots<\e(u_n)$,
\item $b_i=1$ if and only if $q\cdot\e(u_i)=\e(u_{i+1})$.
\end{enumproperties}
\end{definition}

\begin{proposition}\label{prop:compare_in_red_pc}(cf. Prop. 5 in \cite{dlu12efficient} and Section 2.1 in \cite{muw11pc})
Given a reduced circuit and two markings $K$ and $M$, the values $\e(K)$ and $\e(M)$ can be compared (yielding $<$, $=$, or $>$ as the result) in $\O(\abs{\Gamma})$ time. The algorithm can also determine whether $\abs{\e(K)-\e(M)}=1$. 
\end{proposition}

\begin{proof}
Assume that we want to determine whether for a sum $\eps=\sum_{i=0}^n\delta_i\cdot q^i$ with $\abs{\delta_i}\le 2q-2$ we have $\eps\le -2$, $\eps=-1$, $\eps=0$, $\eps=+1$ or $\eps\ge +2$. We can do this inductively using the following procedure:
\begin{enumcases}
\item If $\delta_n=0$, use induction on $\eps=\sum_{i=0}^{n-1}\delta_i\cdot q^i$. 
\item If $\abs{\delta_n}\ge 2$, then $\abs{\eps}\ge 2\cdot q^n-\sum_{i=0}^{n-1}(2q-2)\cdot q^i\ge 2$ and the sign of $\eps$ is the same as the sign of $\delta_n$. 
\item If $\abs{\delta_n}=1$, look at $\delta_{n-1}$. If $\delta_{n-1}=0$ or if it has the same sign as $\delta_n$, then $\abs{\eps}\ge q^n-\sum_{i=0}^{n-2}(2q-2)\cdot q^i=q^{n-1}(q-2)+2\ge 2$ since $q\ge 2$. Again, $\eps$ has the same sign as $\delta_n$. 
If $\delta_{n-1}$ has the opposite sign of $\delta_n$, use induction on $\eps=\hat\delta_{n-1}\cdot q^{n-1}+\sum_{i=0}^{n-2}\delta_i\cdot q^i$, where $\hat\delta_{n-1}=\delta_n\cdot q+\delta_{n-1}\in\{-2q+2,\ldots,2q-2\}$. 
\end{enumcases}

The answer to the original question can be found by applying this algorithm to the mapping $M-K:\Gamma\to\{-2q+2,\ldots,2q-2\}$ given by $(M-K)(u)=M(u)-K(u)$. Note, that the absolute indices $i$ of the $\delta_i$ are not actually needed. Instead one can use the information provided by the reduced circuit. 
\end{proof}

\begin{corollary}\label{cor:divisibility_in_red_pc}
For two markings $K$ and $M$ in a reduced circuit, it can be tested in $\O(\abs{\Gamma})$ time whether $q^{\e(K)}$ divides $\e(M)$. 
\end{corollary}

\begin{proof}
Let $u$ be the node of minimal value in $\supp M$. Then $q^{\e(K)}\mid\e(M)$ if and only if $q^{\e(K)}\mid\e(u)$. Using Proposition \ref{prop:compare_in_red_pc}, we can check the equivalent condition $\e(K)\le\e(\Lambda_u)$. 
\end{proof}

Power circuits arising from a sequence of arithmetic operations are usually far from being reduced. Every cloning creates a pair of nodes with the same value. Therefore, we need an algorithm that given an arbitrary circuit produces an equivalent reduced circuit. In this context, equivalence means that for each node and each marking in the old circuit, there is one with the same value in the reduced circuit. Before giving the algorithm, we need some preparations. 

\begin{definition}\label{def:chain}
A list $u_1\ldots,u_k$ of nodes in a power circuit is called a \emph{chain}\index{power circuit!chain} (starting at $u_1$), if $q\cdot\e(u_i)=\e(u_{i+1})$ for all $1\le i<n$. It is called a maximal chain (starting at $u_1$), if it is not part of a longer chain (starting at $u_1$). 
\end{definition}

Note that chains have nothing to do with paths in the graph. In arbitrary power circuits, chains are difficult to spot. However, in a reduced power circuit, they can be easily identified using the bit vector. 

In a reduced circuit, the maximal chain starting at the unique node of value $1$ is of particular interest. It is called the \emph{base chain}\index{base chain of a power circuit} of the power circuit. For later use, we define in Algorithm \ref{alg:prolong_base_chain} a procedure \PBC prolonging this chain by one node without destroying the reducedness property of a circuit. The procedure \PBC takes $\O(\abs{\Gamma})$ time. 

\begin{algorithm}[ht]\label{alg:prolong_base_chain}
\Input{a reduced power circuit $\Pi=(\Gamma,\delta)$}
\Output{a reduced power circuit $\Pi^\prime=(\Gamma\cupi\{u\},\delta^\prime)$ which is $\Pi$ with an additional node $u$ prolonging the base chain of $\Pi$}
\BlankLine
Let $\Gamma=(v_0,\ldots,v_n)$ be the ordered list of the nodes of the reduced circuit $\Gamma$. Using this list and the bit vector, find the smallest $i\ge 0$ such that $\e(v_i)>q^i$.\;
Write $i$ as a $q$-ary number $i=\sum_{\ell=0}^{i-1}\alpha_\ell\cdot q^\ell$ and use this to define the marking $M(v_\ell)=\alpha_\ell$ with value $\e(M)=i$. 
Insert a new node $u$ with $\Lambda_u=M$ into the circuit.\;
Place $u$ in the ordered list of nodes between $v_{i-1}$ and $v_i$.\;
Check whether $q\cdot\e(u)=\e(v_i)$ by applying Proposition \ref{prop:compare_in_red_pc} to $\Lambda_u$ and $\Lambda_{v_i}$ (both are contained in the reduced circuit $\Pi$). Set the bit vector for $u$ accordingly.\;
\caption{Procedure \PBC}
\end{algorithm}

Now we can give an algorithm that reduces power circuits. Reduction is done node by node. This means that at any point during the reduction procedure, the circuit consists of a reduced part and a part that is not yet reduced. The nodes in the non-reduced part are processed in topological order. In this way, the procedure only has to work for nodes all of whose successors are already in the reduced part. 

This approach allows us to generalize the reduction procedure. Instead of reducing the entire circuit, we can take into account that parts of it might already be reduced. This will turn out to be useful in applications. The procedure \ERed described in Algorithm \ref{alg:extend_reduce} takes as input not only the power circuit but also a list $\mathcal{M}$ of markings that need to be adjusted during reduction in order to preserve their value. 

\begin{algorithm}\label{alg:extend_reduce}
\Input{a graph $\Pi=(\Gamma\cupi U,\delta)$ such that $(\Gamma,\delta\vert_{\Gamma\times\Gamma})$ is a reduced power circuit, a list $\mathcal{M}=(M_1,\ldots M_m)$ of markings in $\Pi$}
\Output{a reduced power circuit $\Pi^\prime=(\Gamma^\prime,\delta^\prime)$ with $\Gamma\subseteq\Gamma^\prime$ and $\delta^\prime\vert_{\Gamma\times\Gamma}=\delta\vert_{\Gamma\times\Gamma}$, a list $\mathcal{M}^\prime=(M_1^\prime,\ldots,M_m^\prime)$ of markings in $\Pi^\prime$ such that $\e(M_i)=\e(M_i^\prime)$}
\BlankLine
Compute a topological order of $U$, i.e., $U=(u_1,\ldots,u_k)$ such that $\delta(u_i,u_j)\neq 0$ implies $i>j$.\;
\For{$i=1,\ldots,k$}{\label{alg:extend_reduce:main_loop}
$U:=U\setminus\{u_i\}$\;
If $\Gamma=\emptyset$, set $\Gamma:=\{u_1\}$ (a circuit with just one node is obviously reduced) and continue with the iteration $i=2$.\;
Let $\Gamma=(v_1,v_2,\ldots)$ be the ordered list of the nodes of the reduced circuit $\Gamma$. Using binary search, find the minimal $j$ such that $\e(u_i)\le\e(v_j)$. Comparing $u_i$ to some $v_j$ is done by comparing $\Lambda_{u_i}$ to $\Lambda_{v_j}$. Both markings are in the reduced part, so Proposition \ref{prop:compare_in_red_pc} applies.\;\label{alg:extend_reduce:bin_search}
\lIf{$\e(\Lambda_{u_i})<0=\e(\Lambda_{v_1})$}{the graph $\Pi$ is not a power circuit; abort the algorithm.}\;\label{alg:extend_reduce:abort}
\eIf{$\e(u_i)<\e(v_j)$ (or no such $v_j$ exists)}{\label{alg:extend_reduce:no_collision}
$\Gamma:=\Gamma\cup\{u_i\}$\;
Insert $u_i$ into $\Gamma$'s sorted list of nodes between $v_{j-1}$ and $v_j$.\;
Set the bit vector for $u_i$ according to whether $\e(\Lambda_{u_i})+1=\e(\Lambda_{v_j})$.\;}
({\ $\e(u_i)=\e(v_j)$}){\label{alg:extend_reduce:collision}
Find the last node $v_k$ of the maximal chain starting at $v_j$ and create $v:=\clone(v_k)$.\;
Multiply the value of $v$ by $q$ by adding $1$ to $\Lambda_{v}$: Let $v_\ell$ be the first node in the base chain with $\Lambda_v(v_\ell)<q-1$. If such $v_\ell$ does not exist, call \PBC to create it. Set $\Lambda_v(v_1)=\ldots=\Lambda_v(v_{\ell-1})=0$ and increment $M(v_\ell)$ by one.\;
Insert $v$ in the ordered list after $v_k$ and set the bit vector for $v$ by comparing $\Lambda_v$ to $\Lambda_{v_{k+1}}$.\;
\ForEach{$M\in\{\supp\Lambda_u\>:\>u\in U\}\cup\mathcal{M}$ with $u_i\in\supp M$}{\label{alg:extend_reduce:adjust_markings}Replace $u_i$ in $M$ by $v_j$, i.e., set $M(v_j):=M(v_j)+M(u_i)$ and $M(u_i):=0$. If now $M(v_j)=\alpha\not\in D$, write $\alpha=\beta\cdot q+\gamma$ with $\gamma\in D$, set $M(v_j):=\gamma$ and add $\beta$ to $M(v_{j+1})$. If again $M(v_{j+1})\not\in D$, repeat. This terminates at the latest at the newly created node $v$ which is not marked by $M$.\;}}}
\caption{Procedure \ERed}
\end{algorithm}

\begin{proposition}\label{prop:extend_reduce}(cf. \cite{dlu12efficient}, Thm. 6)
The procedure \ERed is correct and takes $\sO\left((\abs{\Gamma}+\abs{U})\cdot(\abs{U}+m)\right)$ time. The circuit growth $\abs{\Gamma^\prime\setminus\Gamma}$ is bounded by $2\abs{U}$. 
\end{proposition}

\begin{proof}
At first, a topological order is computed. The time for this is bounded by the size of the subgraph $U$ (nodes and edges), which is $\O(\abs{U}^2)$. In the main loop starting at line \ref{alg:extend_reduce:main_loop}, the nodes are eliminated from $U$ one by one. Let $n=\abs{\Gamma}+\abs{U}$ be the initial size of the whole graph. Since $\Gamma$ grows by $\O(\abs{U})$ during the procedure (although we keep calling it $\Gamma$ for convenience), $\O(n)$ is the correct bound for the size of $\Gamma$. 

For each node $u_i\in U$, its position in the ordering of $\Gamma$ has to be found in step \ref{alg:extend_reduce:bin_search}. Since $u_j$ is chosen to be topologically minimal, the successor marking $\Lambda_u$ is contained in the reduced circuit $\Gamma$, so $u$ can be compared to any node $v\in\Gamma$ in $\O(n)$ time. Using binary search, $\O(\log n)$ comparisons are sufficient, taking $\sO(n\cdot\abs{U})$ time in total. 

For the insertion of $u_i$ in $\Gamma$, we distiguish two cases. In the first one (step \ref{alg:extend_reduce:no_collision}), there is no node in $\Gamma$ with the same value as $u_i$. In this case, $u_i$ is moved from $U$ to $\Gamma$ without any modification. Markings containing $u_i$ (including successor markings, i.e., edges with target $u_i$) are not affected either. 

The second case (step \ref{alg:extend_reduce:collision}), where there is a node $v_j$ with the same value as $u_i$ is more difficult. Figure \ref{fig:pc_reduction} shows an example. The idea is to delete $u_i$ and replace it in all markings $M$ (both markings from $M$ and successor markings of nodes in $U$) by $v_j$. This may cause $v_j$ to by ``overmarked'' by $M$, i.e., $M(v_j)\not\in D$. For example, in the simplest case $q=2$, if $M(v_j)=M(u_i)=1$, then $M(v_j)=2$ after the replacement. The solution is inspired by the idea of carry digits used when adding two $q$-ary numbers: if $\abs{M(v_j)}\ge q$, subtract the appropriate number $\alpha\cdot q$ and add $\alpha$ to the value that $M$ assigns to the next node $v_{j+1}$ in the chain, which has $q$ times the value of $v_j$. The carry might propagate to the end of the chain, which is why we preventively prolonged it by one node $v_{k+1}$. 

Note that the time bound for one execution of step \ref{alg:extend_reduce:adjust_markings} is not $\sO(\abs{\Gamma})$, but rather $\O(\abs{\Gamma}\cdot\#\text{of markings})$. Since this is not sufficient to prove the claimed bound, instead we count the total amount of time spent in step \ref{alg:extend_reduce:adjust_markings} during the whole procedure. The key observation is that for every carry that has to be moved to the next node in the chain, the number $C:=\sum_{M}\sum_{v\in\Gamma\cup U}\abs{M(v)}$ decreases. Initially $C\le (q-1)\cdot\left(n\cdot(\abs{U}+m)\right)$, so the total time complexity of step \ref{alg:extend_reduce:adjust_markings} is $\O(n\cdot(\abs{U}+m))$. 
\end{proof}

\begin{figure}[ht]
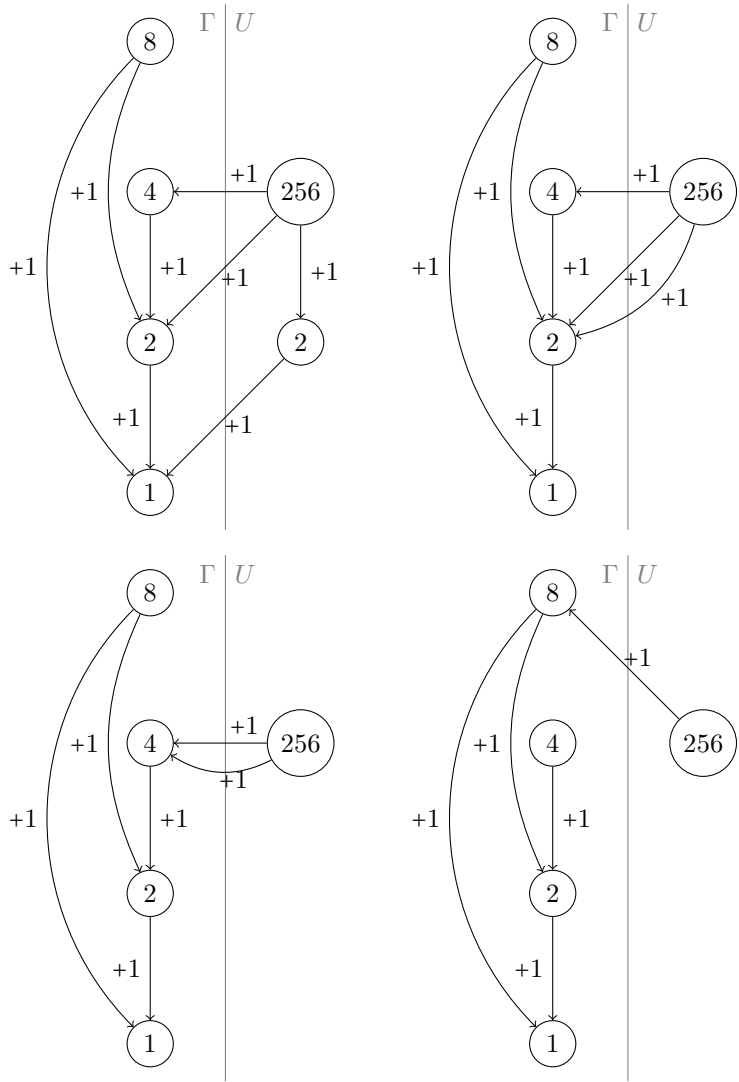

	\begin{center}
		\begin{PCpic}
			\path[use as bounding box] (-2,-0.5) rectangle (3,6.5);
			\draw[gray] (1,-0.5) -- (1,6.5);
			\node[gray,anchor=north east] at (1,6.5) {$\Gamma$};
			\node[gray,anchor=north west] at (1,6.5) {$U$};
			\PCnode{1}{$1$}{(0,0)}
			\PCnode{2}{$2$}{(0,2)}
			\PCnode{4}{$4$}{(0,4)}
			\PCnode{8}{$8$}{(0,6)}
			\PCnode{2a}{$2$}{(2,2)}
			\PCnode{256}{$256$}{(2,4)}
			\PCedge{2}{1}{$+1$}{left}{}
			\PCedge{4}{2}{$+1$}{right}{}
			\PCedge{8}{1}{$+1$}{left}{bend right=45}
			\PCedge{8}{2}{$+1$}{left}{bend right=25}
			\PCedge{2a}{1}{$+1$}{below right=-0.2cm}{}
			\PCedge{256}{2a}{$+1$}{right}{}
			\PCedge{256}{2}{$+1$}{below right=-0.2cm}{}
			\PCedge{256}{4}{$+1$}{above right}{}
		\end{PCpic}
		\begin{PCpic}
			\path[use as bounding box] (-2,-0.5) rectangle (3,6.5);
			\draw[gray] (1,-0.5) -- (1,6.5);
			\node[gray,anchor=north east] at (1,6.5) {$\Gamma$};
			\node[gray,anchor=north west] at (1,6.5) {$U$};
			\PCnode{1}{$1$}{(0,0)}
			\PCnode{2}{$2$}{(0,2)}
			\PCnode{4}{$4$}{(0,4)}
			\PCnode{8}{$8$}{(0,6)}
			\PCnode{256}{$256$}{(2,4)}
			\PCedge{2}{1}{$+1$}{left}{}
			\PCedge{4}{2}{$+1$}{right}{}
			\PCedge{8}{1}{$+1$}{left}{bend right=45}
			\PCedge{8}{2}{$+1$}{left}{bend right=25}
			\PCedge{256}{2}{$+1$}{right}{bend left}
			\PCedge{256}{2}{$+1$}{below right=-0.2cm}{}
			\PCedge{256}{4}{$+1$}{above right}{}
		\end{PCpic}\\
		\vspace{3mm}
		\begin{PCpic}
			\path[use as bounding box] (-2,-0.5) rectangle (3,6.5);
			\draw[gray] (1,-0.5) -- (1,6.5);
			\node[gray,anchor=north east] at (1,6.5) {$\Gamma$};
			\node[gray,anchor=north west] at (1,6.5) {$U$};
			\PCnode{1}{$1$}{(0,0)}
			\PCnode{2}{$2$}{(0,2)}
			\PCnode{4}{$4$}{(0,4)}
			\PCnode{8}{$8$}{(0,6)}
			\PCnode{256}{$256$}{(2,4)}
			\PCedge{2}{1}{$+1$}{left}{}
			\PCedge{4}{2}{$+1$}{right}{}
			\PCedge{8}{1}{$+1$}{left}{bend right=45}
			\PCedge{8}{2}{$+1$}{left}{bend right=25}
			\PCedge{256}{4}{$+1$}{below right=-0.2cm}{bend left}
			\PCedge{256}{4}{$+1$}{above right}{}
		\end{PCpic}
		\begin{PCpic}
			\path[use as bounding box] (-2,-0.5) rectangle (3,6.5);
			\draw[gray] (1,-0.5) -- (1,6.5);
			\node[gray,anchor=north east] at (1,6.5) {$\Gamma$};
			\node[gray,anchor=north west] at (1,6.5) {$U$};
			\PCnode{1}{$1$}{(0,0)}
			\PCnode{2}{$2$}{(0,2)}
			\PCnode{4}{$4$}{(0,4)}
			\PCnode{8}{$8$}{(0,6)}
			\PCnode{256}{$256$}{(2,4)}
			\PCedge{2}{1}{$+1$}{left}{}
			\PCedge{4}{2}{$+1$}{right}{}
			\PCedge{8}{1}{$+1$}{left}{bend right=45}
			\PCedge{8}{2}{$+1$}{left}{bend right=25}
			\PCedge{256}{8}{$+1$}{above right=-0.2}{}
		\end{PCpic}
	\end{center}
	\caption[Reduction step]{Reduction step ($q=2$)}
	\label{fig:pc_reduction}
\end{figure}

\begin{remark}
Not all markings need to be included in $\mathcal{M}$. Since $\Pi$ remains a subcircuit of $\Pi^\prime$ and the values of nodes in $\Gamma$ do not change, all markings whose support is completely contained in $\Gamma$ are automatically preserved. Only markings using nodes in $U$ have to be put into $\mathcal{M}$. In most applications, $\mathcal{M}$ consists only of a constant number of markings. 
\end{remark}

\begin{remark}
The bound for the circuit growth given in Proposition \ref{prop:extend_reduce} is a rather crude one. A more detailed analysis shows that calling \PBC is only necessary once every time $\abs{\Gamma}$ grows by a factor $q$. If one does some ``cleaning up'' in the circuit (for instance delete unmarked nodes with no incoming edges), \cite{muw11pc} shows that the growth during reduction is even bounded by $1$. However, this bound is of no importance in our applications since cloning during arithmetic operations increases the size by $\O(\abs{U})$ anyway. 

In practice, the circuit size rarely ever increases at all during reduction. Usually, the cicuit even shrinks. 
\end{remark}

\begin{theorem}\label{thm:reduce}(cf. \cite{dlu12efficient}, Cor. 7)
There is a procedure \Red which given a power circuit $\Pi=(\Gamma,\delta)$ and a list $\mathcal{M}=(M_1,\ldots,M_m)$ of markings in $\Pi$, returns a reduced circuit $\Pi^\prime=(\Gamma^\prime,\delta^\prime)$ and a list $\mathcal{M}^\prime=(M_1^\prime,\ldots,M_m^\prime)$ of markings in $\Pi^\prime$ such that $\e(M_i)=\e(M_i^\prime)$ ($1\le i\le m$). \Red takes $\sO(\abs{\Gamma}^2+\abs{\Gamma}\cdot m)$ time and the size of $\Gamma^\prime$ is bounded by $2\abs{\Gamma}$. 
\end{theorem}

\begin{proof}
Invoke \ERed taking $\emptyset$ as the reduced part and the whole circuit as $U$. 
\end{proof}

Step \ref{alg:extend_reduce:abort} in \ERed tests whether $\e(\Lambda_u)\ge 0$. This is one of the equivalent conditions specified in Lemma \ref{lem:pc_condition} for a graph to be a power circuit. Therefore, reduction is at the same time a test whether a graph is a power circuit:

\begin{corollary}\label{cor:test_graph_pc}(\cite{dlu12efficient}, Cor. 8)
Given a dag $\Pi=(\Gamma,\delta)$ it can be determined in $\sO(\abs{\Gamma}^2)$ time whether $\Pi$ is a power circuit. \qed
\end{corollary}

\subsection{Compactness}\label{sec:compactness}

In this section, we will show that using a richer data structure for reduced circuits, the time complexity of \ERed can be reduced to $\O((\abs{\Gamma}+\abs{U})\cdot\abs{U})$. This eliminates the logarithmic factors both for \Red and for the test presented in Corollary \ref{cor:test_graph_pc}. We start by taking a closer look at power sums. 

\subsubsection{Compact power sums}\label{sec:compact_power_sums}

A power sum\index{power sum} is a formal sum $S=\sum_{i\ge 0}\alpha_i\cdot q^i$ whose coefficients $\alpha_i$ are in $D$ and only finitely many of them are non-zero. The value $\e(S)\in\Z$ of the power sum $S$ is defined in the obvious way. On the set of all power sums, we define a rewriting system $\mathcal{P}$ generated by the rules
\begin{align*}
	(1)\quad\alpha\cdot q^i+\beta\cdot q^{i+1}&\longrightarrow(\alpha-q)\cdot q^i+(\beta+1)\cdot q^{i+1}&\text{for}\ \alpha>0,\,\beta<0,\cr
	(2)\quad\alpha\cdot q^i+\beta\cdot q^{i+1}&\longrightarrow(\alpha+q)\cdot q^i+(\beta-1)\cdot q^{i+1}&\text{for}\ \alpha<0,\,\beta>0,
\end{align*}
and for $i<j$
\begin{align*}
	(3)\quad&\alpha\cdot q^i+(q-1)\cdot(q^{i+1}+\ldots+q^j)+\beta\cdot q^{j+1}\cr
	&\qquad\longrightarrow(\alpha-q)\cdot q^i+(\beta+1)\cdot q^{j+1}&\text{for}\ \alpha>0,\,\beta<q-1,\cr
	(4)\quad&\alpha\cdot q^i+(-q+1)\cdot(q^{i+1}+\ldots+q^j)+\beta\cdot q^{j+1}\cr
	&\qquad\longrightarrow(\alpha+q)\cdot q^i+(\beta-1)\cdot q^{j+1}&\text{for}\ \alpha<0,\,\beta>-q+1.
\end{align*}
None of these rules changes the value of the power sum. We omit the proof for the following lemma, since it consists of a long but simple enumeration of cases. 

\begin{lemma}\label{lem:power_sum_system_confluent}
The rewriting system $\mathcal{P}$ is locally confluent. \qed
\end{lemma}

\begin{lemma}\label{lem:power_sum_system_terminating}
The rewriting system $\mathcal{P}$ is terminating (noetherian) and hence confluent (\cite{bo93string}, Thm. 1.1.13). 
\end{lemma}

\begin{proof}
Let $S_1\underset{\mathcal{P}}{\Rightarrow}^\ast S_2\underset{\mathcal{P}}{\Rightarrow}^\ast\ldots$ be a sequence of rewritings. Since none of the rules of $\mathcal{P}$ increases the number of non-zero coefficients, this number must eventually reach a minimum. Thus, ignoring a finite number of terms, we can assume that the number of non-zeros is constant within the sequence. No rule in $\mathcal{P}$ moves a non-zero coefficient to the left (in the direction of smaller exponents). As the value of the $S_i$ is fixed, non-zeros cannot be moved indefinitely to the right either. Again, by disregarding a finite prefix of the sequence, we assume that the positions of the non-zero coefficients are fixed. At this point, no application of $(3)$ or $(4)$ is possible any more. Finally, rules of type $(1)$ and $(2)$ move pairs of consecutive coefficients with opposite signs to the left (or remove them), which can also occur only finitely often. Thus, the sequence $S_1,S_2,\ldots$ is eventually constant. 
\end{proof}

We call power sums that are irreducible with respect to $\mathcal{P}$ compact\index{power sum!compact}. If $S=\sum_{i\ge 0}\alpha_iq^i$ is compact, then so is $-S=\sum_{i\ge 0}(-\alpha_i)q^i$. 

\begin{proposition}\label{prop:power_sum_properties}~
\begin{enumproperties}
\item\label{prop:power_sum_properties:unique} For each power sum there is a unique compact power sum of the same value. 
\item\label{prop:power_sum_properties:min} Compact power sums have the minimal number of non-zero coefficients among all power sums of the same value. 
\item\label{prop:power_sum_properties:inc} If $S$ and $T$ are compact power sums, then $\e(T)=\e(S)+1$ if and only if
\begin{align*}
	S&=U\cdot q^{i+1}+\alpha_i\cdot q^i+\sum_{j=0}^{i-1}\alpha_j\cdot q^j\quad\text{and}\cr
	T&=U\cdot q^{i+1}+\beta_i\cdot q^i+\sum_{j=0}^{i-1}\beta_j\cdot q^j
\end{align*}
for some power sum $U$, $\beta_i=\alpha_i+1$, and for each $0\le j<i$ either $\alpha_j=q-1$ and $\beta_j=0$ or $\alpha_j=0$ and $\beta_j=-q+1$. 
\item\label{prop:power_sum_properties:lex} The usual order on $D$ gives rise to a lexicographical order on power sums (where coefficients of higher powers of $q$ are compared before those of lower powers). Restricted to compact power sums, this lexicographical order coincides with the order by values. 
\end{enumproperties}
\end{proposition}

\begin{proof}
For \refenumproperty{prop:power_sum_properties:unique} it suffices to show that for two power sums $S$ and $T$ of the same value, we have $S\underset{\mathcal{P}}{\Leftrightarrow}^\ast T$. This is true, since applying the rules of $\mathcal{P}$ forward or backward, one can turn any power sum into an ordinary $q$-ary number with coefficients from $\{0,\ldots,q-1\}$. (For instance, for positive values of $S$, use rules of type $(1)$ backward and rules of type $(2)$ forward to push negative coefficients to the right.)

Claim \refenumproperty{prop:power_sum_properties:min} follows from the fact that no rule increases the number of non-zero coefficients. 

For the ``if'' part of \refenumproperty{prop:power_sum_properties:inc}, we observe that component-wise subtraction yields
\begin{equation*}
	T-S=q^i-(q-1)\sum_{j=0}^{i-1}q^j,
\end{equation*}
which evaluates to $1$. For the ``only if'' part, let $S=\sum_{\ell\ge 0}\alpha_\ell\cdot q^\ell$ be compact. Consider $S^\prime=\sum_{\ell\ge 1}\alpha_\ell\cdot q^\ell+(\alpha_0+1)$. If $S^\prime$ is a valid power sum (i.e. $\alpha_0+1\in D$) and $S^\prime$ is compact, it already has the desired form (for $i=0$). Otherwise we have one of the following cases:
\begin{enumcases}
\item $S^\prime$ is not a valid power sum, since $\alpha_0=q-1$. Let $k>0$ be the maximum number such that $\alpha_0=\ldots=\alpha_k=q-1$. We transform $S^\prime$ into
\begin{equation*}
	(\alpha_{k+1}+1)\cdot q^{k+1}+\sum_{\ell>k+1}\alpha_\ell\cdot q^\ell
\end{equation*}
and use induction on $(\alpha_{k+1}+1)\cdot q^0+\sum_{\ell>k+1}\alpha_\ell\cdot q^{\ell-k-1}$. 
\item A rule of type $(1)$ can be applied. We have $\alpha_0=0$ and $\alpha_1<0$. Applying the rule gives
\begin{equation*}
	S^\prime\underset{(1)}{\Longrightarrow}\sum_{\ell\ge 2}\alpha_\ell\cdot q^\ell+(\alpha_1+1)\cdot q+(-q+1).
\end{equation*}
Use induction on $\sum_{\ell\ge 2}\alpha_\ell\cdot q^{\ell-1}+(\alpha_1+1)$. 
\item A rule of type $(3)$ can be applied. We have $\alpha_0=0$ and $\alpha_1=\ldots=\alpha_k=q-1$ and $\alpha_{k+1}<q-1$ for some $k\ge 1$. This yields
\begin{equation*}
	S^\prime\underset{(3)}{\Longrightarrow}\sum_{\ell>k+1}\alpha_\ell\cdot q^\ell+(\alpha_{k+1}+1)\cdot q^{k+1}+(-q+1)
\end{equation*}
and induction applies to $\sum_{\ell>k}\alpha_\ell\cdot q^{\ell-k-1}+(\alpha_{k+1}+1)$. 
\end{enumcases}

Finally, \refenumproperty{prop:power_sum_properties:lex} is a consequence of \refenumproperty{prop:power_sum_properties:inc}. 
\end{proof}

The notion of compactness was introduced in \cite{muw11pc} for $q=2$ and subsequently used in \cite{dlu13efficient}. Our definition is a generalization that inherits most of the original characteristics. 

There is, however, one important difference: it is much less obvious for $q>2$ how to make a power sum compact in linear time. In the case $q=2$ it suffices to apply the rules of $\mathcal{P}$ from left to right. Yet, if for instance $q=3$, the application of rule $(1)$ to
\begin{equation*}
	1+q+q^2+q^3-2q^4\underset{(1)}{\Longrightarrow}+q+q^2-2q^3-q^4
\end{equation*}
turns the previously compact prefix $1+q+q^2+q^3$ into the $\mathcal{P}$-reducible sum $1+q+q^2-2q^3$. 

\begin{proposition}\label{prop:compactify_power_sum}
Any power sum $S=\sum_{i=0}^n\alpha_i\cdot q^i$ can be transformed into a compact power sum with the same value in $\O(n)$ time. 
\end{proposition}

\begin{proof}
Any two power sums $S$ and $T$ with $\e(S)=\e(T)$ can be transformed into each other using only replacements of the form
\begin{equation*}
	\alpha\cdot q^{i-1}+\beta\cdot q^i\longrightarrow(\alpha\pm q)\cdot q^{i-1}+(\beta\mp 1)\cdot q^i.\eqno{(\star^\pm_i)}
\end{equation*}
Moreover, at most one application of $(\star^+_i)$ or $(\star^-_i)$ is needed for each $i$. In fact, whether $(\star^+_i)$ or $(\star^-_i)$ or neither is needed, depends only on $S$ and $T$. This can be seen using induction on $i$. 

Let $T=\sum_{i=0}^{n+1}\beta_i\cdot q^i$ be the compact power sum with $\e(S)=\e(T)$. Define $J_i\in\{+1,-1,0\}$ ($1\le i\le n$) depending on whether the replacement $(\star^+_i)$ or $(\star^-_i)$ or neither of them occurs in the sequence $S\underset{(\star^\pm_i)}{\Longrightarrow}^\ast T$. For notational conveniance, we define $J_0=J_{n+2}=0$ and $\alpha_{n+1}=0$. Then we have $\beta_i=\alpha_i+J_i-q\cdot J_{i+1}$ for $0\le i\le n+1$. 

It remains to compute the values $J_i$ ($1\le i\le n+1$). These are the unique solution of the following system of conditions ($0\le i\le n+1$):
\begin{align*}
	(\meddiamond_i)\quad
	&\alpha_{i-1}+J_{i-1}+q\cdot J_i\in D\ \text{and}\ \alpha_i+J_i+q\cdot J_{i+1}\in D\ \text{and}\cr
	&\sign(\alpha_{i-1}+J_{i-1}+q\cdot J_i)\cdot\sign(\alpha_i+J_i+q\cdot J_{i+1})\neq -1\ \text{and}\cr
	&\text{if }\sign(\alpha_{i-1}+J_{i-1}+q\cdot J_i)=\sign(\alpha_i+J_i+q\cdot J_{i+1})=\pm 1,\cr
	&\qquad\text{then}\ \alpha_i+J_i+q\cdot J_{i+1}\neq\pm(q-1).
\end{align*}
For $i=2,\ldots,n+1$ and for all $j,k\in\{-1,0,+1\}$ we define $\mathcal{J}_i[j,k]$ to be the set of possible values for $J_i$, provided that $J_{i-2}=j$ and $J_{i-1}=k$:
\begin{align*}
	\mathcal{J}_i[j,k]:=\{J_i\>:\>&\text{there are $J_1,\ldots,J_{i-3}$ such that}\cr
	&\text{$J_0=0,J_1,\ldots,J_{i-3},J_{i-2}=j,J_{i-1}=k,J_i$ satisfy $(\meddiamond_0),\ldots,(\meddiamond_{i-1})$}\}
\end{align*}
Since $(\meddiamond_{i-1})$ only affects $J_{i-2},J_{i-1},J_i,\alpha_{i-2},\alpha_{i-1}$, the sets $\mathcal{J}_i[\cdot,\cdot]$ can be computed in constant time using $\mathcal{J}_{i-1}[\cdot,\cdot]$. After this, the solution $J_{n+1},\ldots,J_1$ can be read from right to left. 
\end{proof}

\begin{example}\label{ex:compactify_power_sum}
Suppose that $q=3$ and we want to make $S=1+q-2q^2-2q^3+q^4-q^5$ compact. We get:
\begin{center}
\begin{tabular}{lcccccc}\toprule
$\mathcal{J}_i[j,k]$	&$i=2$				&$i=3$				&$i=4$				&$i=5$				&$i=6$				&$i=7$\cr\midrule
$j=-1,\,k=-1$	&$\emptyset$	&$\emptyset$	&$\emptyset$	&$\{0\}$			&$\emptyset$	&$\emptyset$\cr
$j=-1,\,k= 0$	&$\emptyset$	&$\emptyset$	&$\emptyset$	&$\emptyset$	&$\{-1,0\}$		&$\emptyset$\cr
$j=-1,\,k=+1$	&$\emptyset$	&$\emptyset$	&$\emptyset$	&$\emptyset$	&$\emptyset$	&$\emptyset$\cr
$j= 0,\,k=-1$	&$\emptyset$	&$\emptyset$	&$\{-1\}$			&$\emptyset$	&$\emptyset$	&$\emptyset$\cr
$j= 0,\,k= 0$	&$\{0\}$			&$\{-1\}$			&$\emptyset$	&$\emptyset$	&$\emptyset$	&$\{0\}$		\cr
$j= 0,\,k=+1$	&$\{+1\}$			&$\emptyset$	&$\emptyset$	&$\emptyset$	&$\emptyset$	&$\emptyset$\cr
$j=+1,\,k=-1$	&$\emptyset$	&$\emptyset$	&$\emptyset$	&$\emptyset$	&$\emptyset$	&$\emptyset$\cr
$j=+1,\,k= 0$	&$\emptyset$	&$\emptyset$	&$\emptyset$	&$\emptyset$	&$\emptyset$	&$\emptyset$\cr
$j=+1,\,k=+1$	&$\emptyset$	&$\{0\}$			&$\emptyset$	&$\emptyset$	&$\emptyset$	&$\emptyset$\cr
\bottomrule
\end{tabular}
\end{center}
Since $J_7=0$, we deduce from the last column that $J_5=J_6=0$ as well. The only $0$ in the column for $i=6$ with $k=0$ yields $J_4=-1$ and so on. We end up with $J_0=J_1=J_2=0,\,J_3=J_4=-1,\,J_5=J_6=J_7=0$ which results in $T=1+q+q^2-q^5$. 
\end{example}

\subsubsection{Power circuits and trees}\label{sec:pc_treed}

Property \ref{prop:power_sum_properties:lex} of Proposition \ref{prop:power_sum_properties} is the main motivation for the following definition. 

\begin{definition}\label{def:treed_pc}(cf. \cite{dlu13efficient}, Def. 5)
A power $\Pi=(\Gamma,\delta)$ circuit is called \emph{treed}\index{power circuit!treed} (a convenient shorthand for ``reduced and equipped with additional data in the form of a tree''), if
\begin{enumproperties}
\item\label{def:treed_pc:red} $\Pi$ is reduced (i.e., different nodes evaluate to different values and it is equipped with an ordered list $\Gamma=(u_1,\ldots,u_n)$ and a bit vector, see Definition \ref{def:reduced_pc})
\item\label{def:treed_pc:tree} $\Pi$ is equipped with a directed tree $T$ in which each node has up to $\abs{D}$ outgoing edges, labelled with pairwise distinct values from $D$ (and ordered from left to right by increasing values). For each leaf in $T$, the unique path from the root to this leaf must consist of exactly $n=\abs{\Gamma}$ edges. The sequence of labels $\alpha_1,\ldots,\alpha_n$ of such a path (read from the leaf to the root) corresponds to a marking $M:u_i\mapsto\alpha_i$. For all leaves, the power sum $\sum_{i=1}^n\alpha_i\cdot q^i$ (which evaluates to $\e(M)$) must be compact. Any marking represented as a leaf in this way is called compact. 
\item\label{def:treed_pc:comp} All successor markings $\Lambda_u$ ($u\in\Gamma$) are compact. 
\item\label{def:treed_pc:mark} Any node $v\in\Gamma$ that is contained in the support of some marking (compact or not) is not the top node of a maximal chain. 
\item\label{def:treed_pc:lev} For each level in the tree, there is a list containing the nodes of this level. 
\end{enumproperties}
\end{definition}

\begin{example}
Figure \ref{fig:treed_circuit} shows an example of a treed circuit (for $q=2$) alongside the tree containing its markings. Note that the order of the nodes is implicitly given by the lowest level of the tree. 
\end{example}

\begin{figure}
	\begin{center}
		\begin{PCpic}
			\PCnode{u1}{$u_1$}{(0,0)}
			\PCnode{u2}{$u_2$}{(0,2)}
			\PCnode{u3}{$u_3$}{(0,4)}
			\PCnode{u4}{$u_4$}{(2,5)}
			\PCnode{u5}{$u_5$}{(-2,5)}
			\PCedge{u2}{u1}{$+1$}{left}{}
			\PCedge{u3}{u2}{$+1$}{left}{}
			\PCedge{u4}{u1}{$-1$}{right}{bend left}
			\PCedge{u4}{u3}{$+1$}{above}{}
			\PCedge{u5}{u1}{$+1$}{right}{bend right}
			\PCedge{u5}{u3}{$+1$}{above}{}
			\PCmark{u3}{M}{-1}{above}
			\PCmark{u1}{M}{+1}{right}
		\end{PCpic}
		\begin{tikzpicture}
			\node[circle,fill,inner sep=0pt,minimum size=0.15cm] (e) at (0,0) {};
			\node[circle,fill,inner sep=0pt,minimum size=0.15cm] (0) at (0,-1) {};
			\node[circle,fill,inner sep=0pt,minimum size=0.15cm] (00) at (0,-2) {};
			\node[circle,fill,inner sep=0pt,minimum size=0.15cm] (00m) at (-1.5,-3) {};
			\node[circle,fill,inner sep=0pt,minimum size=0.15cm] (000) at (0,-3) {};
			\node[circle,fill,inner sep=0pt,minimum size=0.15cm] (00p) at (3.5,-3) {};
			\node[circle,fill,inner sep=0pt,minimum size=0.15cm] (00m0) at (-1.5,-4) {};
			\node[circle,fill,inner sep=0pt,minimum size=0.15cm] (0000) at (0,-4) {};
			\node[circle,fill,inner sep=0pt,minimum size=0.15cm] (000p) at (2,-4) {};
			\node[circle,fill,inner sep=0pt,minimum size=0.15cm] (00p0) at (3.5,-4) {};
			\node[circle,fill,inner sep=0pt,minimum size=0.15cm] (00m0p) at (-1,-5) {};
			\node[circle,fill,inner sep=0pt,minimum size=0.15cm] (00000) at (0,-5) {};
			\node[circle,fill,inner sep=0pt,minimum size=0.15cm] (0000p) at (1,-5) {};
			\node[circle,fill,inner sep=0pt,minimum size=0.15cm] (000p0) at (2,-5) {};
			\node[circle,fill,inner sep=0pt,minimum size=0.15cm] (00p0m) at (3,-5) {};
			\node[circle,fill,inner sep=0pt,minimum size=0.15cm] (00p0p) at (4,-5) {};
			\draw[->]	(e) edge node[right] {$0$} (0)
								(0) edge node[right] {$0$} (00)
								(00) edge node[above left=-0.1] {$-1$} (00m)
								(00) edge node[right] {$0$} (000)
								(00) edge node[above right=-0.1] {$+1$} (00p)
								(00m) edge node[right] {$0$} (00m0)
								(000) edge node[right] {$0$} (0000)
								(000) edge node[above right=-0.1] {$+1$} (000p)
								(00p) edge node[right] {$0$} (00p0)
								(00m0) edge node[right] {$+1$} (00m0p)
								(0000) edge node[left] {$0$} (00000)
								(0000) edge node[right] {$+1$} (0000p)
								(000p) edge node[left] {$0$} (000p0)
								(00p0) edge node[left] {$-1$} (00p0m)
								(00p0) edge node[right] {$+1$} (00p0p);
			\node[gray] (u1) at (-2,-4.5) {$u_1$};
			\node[gray] (u2) at (-2,-3.5) {$u_2$};
			\node[gray] (u3) at (-2,-2.5) {$u_3$};
			\node[gray] (u4) at (-2,-1.5) {$u_4$};
			\node[gray] (u5) at (-2,-0.5) {$u_5$};
			\node (li0) at (-2.5,0) {};
			\node (li1) at (-2.5,-1) {};
			\node (li2) at (-2.5,-2) {};
			\node (li3) at (-2.5,-3) {};
			\node (li4) at (-2.5,-4) {};
			\node (li5) at (-2.5,-5) {};
			\node (lt0) at (5,0) {};
			\node (lt1) at (5,-1) {};
			\node (lt2) at (5,-2) {};
			\node (lt3) at (5,-3) {};
			\node (lt4) at (5,-4) {};
			\node (lt5) at (5,-5) {};
			\draw[gray,->]	(li0) edge (e) (e) edge (lt0)
											(li1) edge (0) (0) edge (lt1)
											(li2) edge (00) (00) edge (lt2)
											(li3) edge (00m) (00m) edge (000) (000) edge (00p) (00p) edge (lt3)
											(li4) edge (00m0) (00m0) edge (0000) (0000) edge (000p) (000p) edge (00p0) (00p0) edge (lt4)
											(li5) edge (00m0p) (00m0p) edge (00000) (00000) edge (0000p) (0000p) edge (000p0) (000p0) edge (00p0m) (00p0m) edge (00p0p) (00p0p) edge (lt5);
			\node[anchor=north] (M) at (-1,-5) {$M$};
			\node[anchor=north] (Lv1) at (0,-5) {$\Lambda_{v_1}$};
			\node[anchor=north] (Lv2) at (1,-5) {$\Lambda_{v_2}$};
			\node[anchor=north] (Lv3) at (2,-5) {$\Lambda_{v_3}$};
			\node[anchor=north] (Lv4) at (3,-5) {$\Lambda_{v_4}$};
			\node[anchor=north] (Lv5) at (4,-5) {$\Lambda_{v_5}$};
		\end{tikzpicture}
	\end{center}
	\caption{Treed power circuit ($q=2$) with compact marking $M$}
	\label{fig:treed_circuit}
\end{figure}
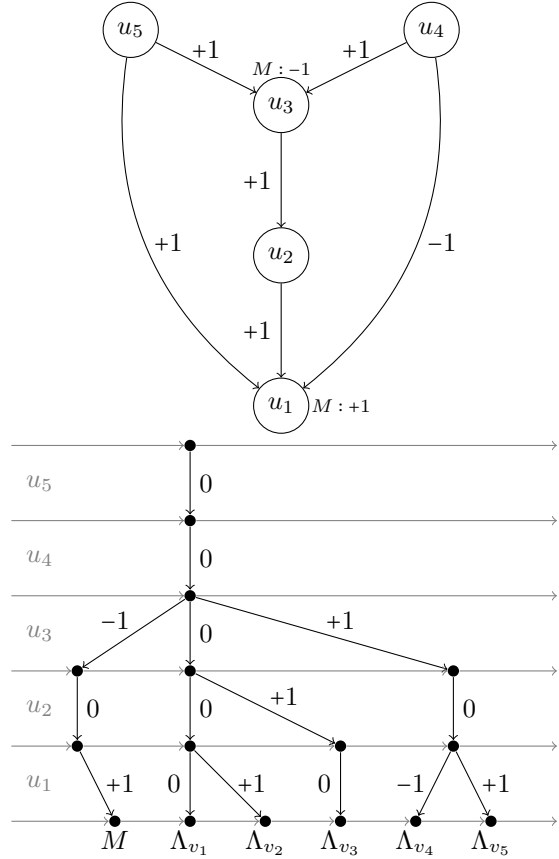

The most time consuming step in \ERed was to find the position of a new node in the sorted list of $\Gamma$. Using binary search, this took $\O(n\log n)$ time. If $\Gamma$ is treed and $\Lambda_u$ is compact, we can improve this to $\O(n)$: the position of the leaf corresponding to $\Lambda_u$ already tells the position of $u$ in $\Gamma$. In order to adjust the bit vector, we have to read the paths from the root of the tree to the respective leaves and check the condition given by Proposition \ref{prop:power_sum_properties} \ref{prop:power_sum_properties:inc}. Yet, making a circuit treed is more complicated than just reducing it. 

For the time analysis, we use amortization\index{amortized analysis} with respect to a potential function $\pot$, mapping power circuits to numbers, see 17.3 in \cite{cormen09introduction}. An algorithm on power circuits is said to run in amortized time $t$, if the real running time is bounded by $t+\pot(\Pi)-\pot(\Pi^\prime)$, where $\Pi$ is the input and $\Pi^\prime$ is the resulting circuit. Thus, an algorithm may take longer than its indicated amortized time, as long as it decreases the potential by the same amount. The potential can be thought of as a debt which is accumulated whenever the algorithm does not terminate in time. 

\begin{definition}
For a power circuit $\Pi=(\Gamma,\delta)$, the number of maximal chains is denoted $\ch(\Pi)$. The \emph{potential} of $\Pi$\index{potential of a power circuit} is $\pot(\Pi)=\ch(\Pi)\cdot\abs{\Gamma}$. 

In a situation where $\Pi=(\Gamma\cupi U,\delta)$ is a graph with an embedded power circuit $\Pi^\prime=(\Gamma,\delta\vert_{\Gamma\times\Gamma})$, we define $\ch(\Pi):=\ch(\Pi^\prime)$ and $\pot(\Pi):=\pot(\Pi^\prime)$. 
\end{definition}

\begin{lemma}\label{lem:insert_node}
There is a procedure \InsNode, which takes a treed power circuit $\Pi=(\Gamma,\delta)$ and a compact marking $M$ in $\Pi$ and which inserts a new node $u$ with $\Lambda_u=M$ into $\Pi$. \InsNode runs in amortized time $\O(\abs{\Gamma})$. 
\end{lemma}

\begin{proof}
Since $M$ is compact, the position of $v$ in the ordered list of nodes is determined by the position of the leaf corresponding to $M$ in the tree. The bit vector can be adjusted by comparing $M$ to the successor markings of the nodes immediately before and after $v$. 

The tree has to be ``stretched'' by inserting a new level corresponding to $v$. All edges on this level are labelled $0$, since no marking contains the new node $v$. Using the lists of nodes of the same level, this takes $\O(\abs{\Gamma})$ time. 

The insertion of $v$ increases $\abs{\Gamma}$ and in some cases also $\ch(\Pi)$. Thus, the potential grows by up to $\abs{\Gamma}+\ch(\Pi)+1\in\O(\abs{\Gamma})$. 
\end{proof}

Incrementing a compact marking by $1$ is easier than doing the same with an arbitrary marking. Let $M$ be a compact marking in a treed circuit, in which $u_1,u_2$ are the unique nodes with values $1$ and $2$ respectively (assuming they exist). If $M(u_1)<q-1$, $M(u_1)$ can be incremented directly. If $M(u_1)=q-1$, then $M(u_2)<q-1$ thanks to compactness. Set $M(u_1):=0$ and increment $M(u_2)$. However, the resulting marking is not always compact. 

\begin{lemma}\label{lem:compactify_marking}
There is a procedure \CompMark which, given an arbitrary marking $M$ in a treed circuit $\Pi=(\Gamma,\delta)$, computes a compact marking $M^\prime$ with $\e(M^\prime)=\e(M)$ in $\O(\abs{\Gamma})$ time. \end{lemma}

\begin{proof}
This is based on the following observation: if $S=\sum_{i=n}^m\alpha_i\cdot q^i$ is a power sum, then in the corresponding compact power sum $T$ only the coefficients of $q^{n},\ldots,q^{m+1}$ can be non-zero. Hence, for each individual chain $C\subseteq\Gamma$, we can apply the algorithm from Proposition \ref{prop:compactify_power_sum} to the power sum defined by $M\vert_C$. Note that a node for $q^{m+1}$ exists due to Property \ref{def:treed_pc:mark} of the treed circuit $\Pi$. 
\end{proof}

Strictly speaking, we have not proved that the resulting circuit is treed. In fact, Condition \ref{def:treed_pc:mark} demanding an unmarked node at the top of every chain might have been lost during compactification. We will fix this shortly. For now, we content ourselves with the observation that the problem does not arise when \CompMark is called during \PBC. In this special case, the successor marking of the new node with value $q^i$ only uses nodes with much smaller values from the base chain (approximately the first $\log_q i$). This shows:

\begin{corollary}\label{cor:prolong_base_chain_treed}
The procedure \PBC given in Algorithm \ref{alg:prolong_base_chain} can be adapted for treed power circuits. The amortized time complexity is $\O(\abs{\Gamma})$. \qed
\end{corollary}

\begin{corollary}\label{cor:increment_compact_marking}
There is a procedure \IncMark which, given a compact marking $M$ in a treed power circuit $\Pi=(\Gamma,\delta)$, computes a compact marking $M^\prime$ with $\e(M^\prime)=\e(M)+1$ and leaves $\Pi$ treed. \IncMark takes $\O(\abs{\Gamma})$ amortized time and increases the circuit size by $1+\ch(\Pi)-\ch(\Pi^\prime)$. 
\end{corollary}

\begin{proof}
As discussed above, incrementing $M$ by $1$ only affects the two nodes with values $1$ and $2$. As a consequence, the compactification process of the incremented marking is limited to the base chain of $\Pi$. 

In order to fulfill Condition \ref{def:treed_pc:mark} of Definition \ref{def:treed_pc}, we invoke \PBC which creates a new node $u$. If $u$ is the new top node of the base chain, we are done. Otherwise the insertion of $u$ has linked two maximal chains, decreasing $\ch(\Pi)$ by one. We use the released potential to pay for the $\O(\abs{\Gamma})$ time used so far and repeat. 
\end{proof}

For treed circuits, the procedure \ETree given in Algorithm \ref{alg:extend_tree} replaces \ERed. 

\begin{algorithm}\label{alg:extend_tree}
\Input{a graph $\Pi=(\Gamma\cupi U,\delta)$ such that $(\Gamma,\delta\vert_{\Gamma\times\Gamma})$ is a treed power circuit, a list $\mathcal{M}=(M_1,\ldots M_m)$ of markings in $\Pi$ with $\supp M_i\not\subseteq\Gamma$}
\Output{a treed power circuit $\Pi^\prime=(\Gamma^\prime,\delta^\prime)$ with $\Gamma\subseteq\Gamma^\prime$ and $\delta^\prime\vert_{\Gamma\times\Gamma}=\delta\vert_{\Gamma\times\Gamma}$, a list $\mathcal{M}^\prime=(M_1^\prime,\ldots,M_m^\prime)$ of compact markings in $\Pi^\prime$ such that $\e(M_i)=\e(M_i^\prime)$}
\BlankLine
Compute a topological order $U=(u_1,\ldots,u_k)$.\;
\For{$i=1,\dots,k$}
{
	$U:=U\setminus\{u_i\}$\;
	\lIf{$\Gamma=\emptyset$}{set $\Gamma:=\{u_1\}$, create a tree for $u_1$ and insert all markings with support $\{u_1\}$. Continue with $i=2$.}\;
	Traversing the lowest level of the tree, find the first node $v_j$ in the ordered list $\Gamma=(v_1,v_2,\ldots)$ such that $\e(u_i)\le\e(v_j)$.\label{alg:extend_tree:search}\;
	\lIf{$\e(\Lambda_{u_i})<0$}{abort the algorithm.}\;
	\eIf{$\e(u_i)<e(v_j)$ (or no such $v_j$ exists)}
	{
		Insert $u_i$ into $\Gamma$ using \InsNode.\;
	}
	({\ $\e(u_i)=\e(v_j)$})
	{
		Prolong the maximal chain $v_j,v_{j+1},\ldots,v_k$ by a new node by creating a copy of $\Lambda_{v_k}$ and applying \IncMark and \InsNode.\;\label{alg:extend_tree:prolong_chain}
		\ForEach{marking $M$ with $u_i\in\supp M$}
		{
			Replace $u_i$ in $M$ by $v_j$ by adding $M(u_i)$ to $M(v_j)$ and setting $M(u_i):=0$. If, after this, $M(v_j)\not\in D$, add $\pm q$ to $M(v_j)$ and decrease $M(v_{j+1})$ by $\pm 1$. Repeat if $M(v_{j+1})\not\in D$ and so on.
		}
	}
	\lForEach{marking $M\in\mathcal{M}\cap\{\Lambda_u\>:\>u\in U\}$ of which $u_i$ was the last node not in $\Gamma$\label{alg:extend_tree:comp_marks}}{$\CompMark(M)$.}\;
	\Repeat{the newly created node is the top of the chain}
	{\label{alg:extend_tree:restore_inv_top}
		Prolong the maximal chain starting at $v_j$ using \IncMark and \InsNode
	}
}
\caption{Procedure \ETree}
\end{algorithm}

\begin{proposition}\label{prop:extend_tree}
The procedure \ETree is correct and runs in amortized time $\O((\abs{\Gamma}+\abs{U})\cdot(\abs{U}+m))$. The circuit growth $\abs{\Gamma^\prime\setminus\Gamma}$ is bounded by $4\cdot\abs{U}+\ch(\Pi)-\ch(\Pi^\prime)$. 
\end{proposition}

\begin{proof}
The basic structure of \ETree is the same as that of \ERed, so we focus on the differences. Let $n=\abs{\Gamma}+\abs{U}$. 
Between cycles of the main loop, we keep up the following invariants for all markings $M\in\mathcal{M}\cup\{\Lambda_u\>:\>u\in U\}$:
\begin{enuminvariants}
\item\label{prop:extend_tree:inv_comp} If the support of $M$ is completely contained in $\Gamma$, $M$ is compact. 
\item\label{prop:extend_tree:inv_top} $\Gamma$ is treed. In particular, Condition \ref{def:treed_pc:mark} holds, which means that for all nodes $v\in\supp M\cap\Gamma$, the top node of the maximal chain starting at $v$ is not marked by $M$. 
\end{enuminvariants}
At the beginning, both invariants are true by definition. 

The time complexity for finding $v_j$ in step \ref{alg:extend_tree:search} is reduced to $\O(n)$ due to the representation of the markings in the tree. 

If $\Gamma$ contains no node with the same value as $u_i$, we can insert it as we did in \ERed. Remember that $\Lambda_{u_i}$ is compact due to \refenuminvariant{prop:extend_tree:inv_comp}. The case when $\e(u_i)=\e(v_j)$ also ressembles \ERed and has the same (now amortized) time bound. 

All markings $M$ whose support is completely contained in $\Gamma$ after the processing of $u_i$ must be made compact in order to regain \refenuminvariant{prop:extend_tree:inv_comp}. This is done in step \ref{alg:extend_tree:comp_marks}. After that, we have to restore Condition \ref{def:treed_pc:mark}. Using the same argument as in Corollary \ref{cor:increment_compact_marking}, step \ref{alg:extend_tree:restore_inv_top} takes $\O(n)$ amortized time and causes the circuit to grow by at most $2+\ch(\Pi)-\ch(\Pi^\prime)$ nodes. Together with step \ref{alg:extend_tree:prolong_chain}, the overall circuit growth during \ETree is bounded by $4\cdot\abs{U}+\ch(\Pi)-\ch(\Pi^\prime)$. 
\end{proof}

\begin{theorem}\label{thm:make_tree}(Treed analogon of Theorem \ref{thm:reduce})
There is a procedure \MTree which given a power circuit $\Pi=(\Gamma,\delta)$ and a list $\mathcal{M}=(M_1,\ldots,M_m)$ of markings in $\Pi$, returns a treed circuit $\Pi^\prime=(\Gamma^\prime,\delta^\prime)$ and a list $\mathcal{M}^\prime=(M_1^\prime,\ldots,M_m^\prime)$ of compact markings in $\Pi^\prime$ such that $\e(M_i)=\e(M_i^\prime)$ ($1\le i\le m$). \MTree takes $\O(\abs{\Gamma}^2+\abs{\Gamma}\cdot m)$ time and the size of $\Gamma^\prime$ is bounded by $4\abs{\Gamma}$. \qed
\end{theorem}

\begin{remark}\label{rem:pc_strategy}(Working with treed power circuits)\\
The usual strategy when solving a problem using power circuits is to create one power circuit and keep all integers as markings in this circuit. The power circuit is kept in treed form in order that comparisons can be done efficiently at any time. For each arithmetic operation, the markings corresponding to the operands are cloned, and the operation (addition or multiplication by a power of two) is performed on the clones. Finally, \ETree is called with the set of clones as $U$ to regain a treed circuit. This takes $\O((\abs{\Gamma}+\abs{U})\cdot\abs{U})$ time and this time bound also absorbs everything else done during the operation. 

For an estimate of the time complexity of an entire algorithm, we need to keep track of the circuit size $\abs{\Gamma}$ as well as the size $\omega$ of the supports of the markings. The latter is usually called the ``weight'' of the ciruit and determines the growth during each operation. The cost for one operation is $\O((\abs{\Gamma}+\omega)\cdot\omega)$. If we start with a comparatively small circuit and $\omega$ remains constant during the algorithm (which is normally the case), then after $s$ operations the circuit size is bounded by $\O(s\cdot\omega)$ and the time by $\O(s^2\omega^2)$. In our main application -- the solution of the word problem in Higman's group -- $s$ will turn out to be quadratic and $\omega$ linear in the input size $n$, leading to an $\O(n^6)$ time algorithm. 
\end{remark}

\begin{remark}
Seen from the outside, \ERed and \ETree as well as \Red and \MTree behave very much alike. In applications, all four procedures are used as ``black boxes'' and of the resulting circuits only the weaker property of reducedness is used. Therefore, in order to simplify nomenclature, we will speak of ``reduction'' and ``reduced'' circuits, subsuming both concepts. The reader may then choose whether to use the simpler reduction concept at the cost of logarithmic factors or to go through the more complicated procedures for treed power circuits with better asymptotic time complexity. 
\end{remark}

\section{The Word Problem in Generalized Baumslag-Gersten groups}\label{sec:BG}

Although the main goal of this paper is to solve the word problem in Higmans' groups, we sidetrack briefly to present another generalization that is made possible by power circuits with arbitrary base $q$. 

The Baumslag-Gersten group is defined as
\begin{align*}
	\BG{2}&=\gr{a,b}{a^{a^b}=a^2}\cr
	&=\gr{a,b}{(bab^{-1})a(bab^{-1})^{-1}=a^2}\cr
	&\simeq\gr{a,b,t}{tat^{-1}=a^2,\,bab^{-1}=t}.
\end{align*}

This is an HNN extension of $\BS(1,2)$ generated by $a$ and $t$. Replacing $\BS(1,2)$ by $\BS(1,q)$ (for $q\ge 2$), we get a family of generalized Baumslag-Gersten groups:
\begin{align*}
	\BG{q}&=\gr{a,b}{a^{a^b}=a^q}\cr
	&\simeq\gr{a,b,t}{tat^{-1}=a^q,\,bab^{-1}=t}
\end{align*}

\begin{theorem}\label{thm:wp_BG}
	The word problem for the generalized Baumslag-Gersten group $\BG{q}$ is solvable in $\O(n^3)$ time. \qed
\end{theorem}

The proof of Theorem \ref{thm:wp_BG} is literally the same as that for $\BG{2}$ which was given in \cite{dlu12efficient} and \cite{dlu13efficient}, except that the new base $q$ power circuits from Section \ref{sec:pc} are used. 

\section{The Word Problem in Generalized Higman groups}\label{sec:higman}

We generalize the groups $H_f$ defined in the introduction by replacing the underlying Baumslag-Solitar group $\BS(1,2)$ by $\BS(1,q)$. 

\begin{definition}\label{def:higman_group}
The (generalized) Higman group $\Hig{q}{f}$ is defined as
\begin{equation}
	\Hig{q}{f}=\gr{a_1,\ldots,a_f}{a_{i+1}a_ia_{i+1}^{-1}=a_i^q\ (i\in\Z/f\Z)}.
\end{equation}
\end{definition}

While $H_f=\Hig{2}{f}$ ($f>4$) retains all the important properties of $H_4$ (infinite, huge compression, no non-trivial normal subgroup of finite index), this is not entirely true for $\Hig{q}{f}$ in general. For example, for all $f\ge 1$, the homomorphism given by
\begin{align*}
	\Hig{3}{f}&\twoheadrightarrow\Z/2\Z;\cr
	a_1&\mapsto 1,\cr
	a_2,\ldots,a_f&\mapsto 0
\end{align*}
sends $\Hig{3}{f}$ onto a finite non-trivial group. 

In this section, we will prove:

\begin{theorem}\label{thm:wp_Hig}
Let $q\ge 2,f\ge 4$. The word problem for the generalized Higman group $\Hig{q}{f}$ can be solved in $\O(n^6)$ time. 
\end{theorem}

The key observation for the solution of the word problem is the decomposition of $\Hig{q}{f}$ into a series of amalgamations of $f$ copies of the Baumslag-Solitar group $\BS(1,q)=\gr{a,t}{tat^{-1}=t^q}$, see \cite{serre02trees}. 

\begin{equation*}
	\Hig{q}{f}=G_{1,\ldots,f-1}\ast_{F_{1,f-1}}G_{f-1,f,1},
\end{equation*}
where
\begin{align*}
	G_{1,\ldots,f-1}&=\gr{a_1,\ldots,a_{f-1}}{a_{i+1}a_ia_{i+1}^{-1}=a_i^q (1\le i<f-1)}\ \text{and}\cr
	G_{f-1,f,1}&=\gr{a_{f-1},a_f,a_1}{a_fa_{f-1}a_f^{-1}=a_{f-1}^q,\,a_1a_fa_1^{-1}=a_f^q}
\end{align*}
and in both cases $F_{1,f-1}$ is the subgroup generated by $a_1$ and $a_{f-1}$, which in fact freely generate $F_{1,f-1}$ (if $f\ge 4$). Furthermore, we can break $G_{1,\ldots,f-1}$ and $G_{f-1,f,1}$ down to
\begin{align*}
	G_{1,\ldots,f-1}&=G_{1,2}\ast_{F_2}G_{2,3}\ast_{F_3}\ldots\ast_{F_{f-2}}G_{f-2,f-1}\text{ and}\cr
	G_{f-1,f,1}&=G_{f-1,f}\ast_{F_f}G_{f,1},
\end{align*}
where $G_{i,i+1}=\gr{a_i,a_{i+1}}{a_{i+1}a_ia_{i+1}^{-1}=a_i^2}$ and the indices are read in $\Z/f\Z$. 

Each group $G_{i,i+1}$ is a copy of the Baumslag-Solitar group $\BS(1,q)$ and thus isomorphic to the semidirect product $\Z[1/q]\rtimes\Z$ which consists of pairs $(u,k)\in\Z[1/q]\rtimes\Z$. The isomorphism is given by $a_i\mapsto(1,0)$ and $a_{i+1}\mapsto(0,1)$. In $\Z[1/q]\rtimes\Z$, we have the following formulae for multiplication and inversion:
\begin{align*}
	(u,k)(v,\ell)&=(u+v\cdot q^k,k+\ell)\cr
	(u,k)^{-1}&=(-u\cdot q^{-k},-k)
\end{align*}

When dealing with more than one group $G_{i,i+1}$, we add $i$ as a subscript to those pairs designating an element of $G_{i,i+1}$. 

In order to solve the word problem for $\Hig{q}{f}$, we first need a solution for the subgroup membership problem of $F_{1,e}$ in $G_{1,\ldots,e}$ (with $e\ge 3$; this covers both $G_{1,\ldots,f-1}$ and $G_{f-1,f,1}$). Furthermore, we have to do this in an effective way, i.e., given a sequence of pairs $(u,k)_i$ which represents an element of $F_{1,e}$, we have to find a corresponding sequence of pairs of the form $(u,0)_1$ and $(0,\ell)_{e-1}$. 

We start by giving a reduction system $\mathcal{L}$ for $G_{1,\ldots,e}$:
\begin{align*}
(1)\quad&(u,k)_i(v,\ell)_i\longrightarrow(u+v\cdot q^k,k+\ell)_i&\text{for $1\le i\le e$}\cr
(2)\quad&(u,k)_i(v,0)_{i+1}\longrightarrow(u,k+v)_i&\text{for $1\le i<e$ and $v\in\Z$}\cr
(3)\quad&(u,0)_{i+1}(v,\ell)_i\longrightarrow(v\cdot q^u,\ell+u)_i&\text{for $1\le i<e$ and $u\in\Z$}\cr
(4)\quad&(u,k)_{i+1}(0,\ell)_i\longrightarrow(u+\ell\cdot q^k,k)_{i+1}&\text{for $1\le i<e$}\cr
(5)\quad&(0,k)_i(v,\ell)_{i+1}\longrightarrow(k+v,\ell)_{i+1}&\text{for $1\le i<e$}
\end{align*}

The system $\mathcal{L}$ is not confluent in general, but the following property holds:
\begin{proposition}\label{prop:britton_G_1_f}
If $w$ is an $\mathcal{L}$-reduced word that equals $1$ in $G_{1,\ldots,e}$, then $w$ is the empty word. \qed
\end{proposition}
This ressembles Britton's Lemma for HNN extensions. In fact, Bass-Serre theory provides a unifying notion (and proof) for both phenomena. We give no further proof here, but instead apply the system $\mathcal{L}$ to the subgroup membership problem. 

Let $\mathcal{L}^\prime$ be the system $\mathcal{L}$ extended by the rules
\begin{align*}
(6)\quad&(x_1,-x_2)_1(x_2,-x_3)_2\ldots(x_{e-2},-x_{e-1})_{e-2}(\widetilde x_{e-1},x_e)_{e-1}\cr
&\qquad\longrightarrow(x_1,0)_1(\widetilde x_{e-1}-x_{e-1},x_e)_{e-1}\qquad\text{and}\cr
(7)\quad&(-x_{e-1}\cdot q^{x_e},x_e)_{e-1}(-x_{e-2}\cdot q^{x_{e-1}},x_{e-1})_{e-2}\ldots(-x_2\cdot q^{x_3},x_3)_2(x_1\cdot q^{x_2},\widetilde x_2)_1\cr
&\qquad\longrightarrow(0,x_e)_{e-1}(x_1,\widetilde x_2-x_2)_1,
\end{align*}
where all $x_i\neq 0$. 

The new rules respect the group structure, and hence Proposition \ref{prop:britton_G_1_f} holds for $\mathcal{L}^\prime$ as well. The new rules are not length-increasing, since $e\ge 3$. 

Starting with an arbitrary sequence $w$ of pairs $(u,k)_i$ representing an element in $G_{1,\ldots,e}$, one can compute an equivalent $\mathcal{L}^\prime$-reduced word $\hat w$ with linearly many operations: First, compute an $\mathcal{L}$-reduced word $\tilde w$, then apply rules $(6)$ and $(7)$. Note that the latter leave $\tilde w$ $\mathcal{L}$-reduced. Only the second pair generated by either of these rules can be part of another application of $(6)$ or $(7)$. Therefore, $\tilde w$ can be $\mathcal{L}^\prime$-reduced with one pass from left to right. 

\begin{proposition}
Let $w$ be a sequence of pairs $(u,k)_i$ which represents an element of the subgroup $F_{1,e}\le G_{1,\ldots,e}$. If $w$ is $\mathcal{L}^\prime$-reduced, then $w$ is already an alternating sequence of pairs of types $(u,0)_1$ and $(0,\ell)_{e-1}$. 
\end{proposition}

\begin{proof}
Let $w\streq(u_1,k_1)_{i_1}(u_2,k_2)_{i_2}\ldots(u_n,k_n)_{i_n}$. We assume that
\begin{equation*}
	w\grpeq\tilde w\streq(v_1,0)_1(0,\ell_1)_{e-1}(v_2,0)_1(0,\ell_2)_{e-1}\ldots\in F_{1,e}
\end{equation*}
and that $\tilde w$ contains no trivial pairs $(0,0)_i$ which makes $\tilde w$ $\mathcal{L}$-reduced. The case where $\tilde w$ starts with $(0,\ell_1)_{e-1}(v_1,0)_1\ldots$ is similar. The sequence
\begin{equation*}
	\tilde w^{-1}w\streq\ldots(-v_2,0)_1(0,-\ell_1)_{e-1}(-v_1,0)_1(u_1,k_1)_{i_1}(u_2,k_2)_{i_2}(u_3,k_3)_{i_3}\ldots
\end{equation*}
equals $1$ in $G_{1,\ldots,e}$ and must therefore $\mathcal{L}$-reduce to the empty sequence. Note that both $\tilde w^{-1}$ and $w$ are $\mathcal{L}$-reduced, so any $\mathcal{L}$-reduction can only occur at the border between the two words. 

Clearly, we cannot have $i_1\ge 3$ or else $\tilde w^{-1}w$ would be $\mathcal{L}$-reduced. If $i_1=2$, then a reduction of type $(2)$ is possible if $k_1=0$ and $u_1\in\Z$, in which case we get $(-v_1,0)_1(u_1,0)_2\overset{(2)}{\longrightarrow}(-v_1,u_1)_1$. But after that, the sequence is $\mathcal{L}$-reduced since $k_1=0$ and $u_1\in\Z$ imply $i_2\neq 1$. 

Hence, we are left with $i_1=1$. In that case, we get $(-v_1,0)_1(u_1,k_1)_1\overset{(1)}{\longrightarrow}(-v_1+u_1,k_1)_1$. If this is $(0,0)_1$, we have $(u_1,k_1)_{i_1}=(v_1,0)_1$ and we proceed inductively with the remaining sequence. Otherwise, we must have $u_1=v_1$ in order to continue applying rules. If $e\ge 4$, the next rule can only apply to $(0,k_1)_1(u_2,k_2)_{i_2}$, so $i_2=2$ and we get $(0,k_1)_1(u_2,k_2)_2\overset{(5)}{\longrightarrow}(k_1+u_2,k_2)_2$. Again, the sequence is $\mathcal{L}$-reduced unless $u_2=-k_1$. We iterate this argument until we arrive at
\begin{equation*}
	\tilde w^{-1}w\overset{\ast}{\underset{\mathcal{L}}{\Longrightarrow}}
	\ldots(-v_2,0)_1(0,-\ell_1)_{e-1}(0,k_{e-2})_{e-2}(u_{e-1},k_{e-1})_{i_{e-1}}\ldots.
\end{equation*}
On the way, we have found $u_1=v_1,u_2=-k_1,u_3=-k_2,\ldots,u_{e-2}=-k_{e-3}$, and $i_j=j$ for $1\le j\le e-2$. One further reduction of type $(4)$ brings us to
\begin{equation*}
	\tilde w^{-1}w\overset{\ast}{\underset{\mathcal{L}}{\Longrightarrow}}
	\ldots(-v_2,0)_1(k_{e-2}\cdot q^{-\ell_1},-\ell_1)_{e-1}(u_{e-1},k_{e-1})_{i_{e-1}}\ldots.
\end{equation*}
Since $\ell_1\neq 0$, the next reduction requires $i_{e-1}=e-1$. Thus, rule (6) of $\mathcal{L}^\prime$ can be applied to the prefix $(u_1,k_1)_1\ldots(u_{e-1},k_{e-1})_{e-1}$ of the original word $w$. 
\end{proof}

For the amalgamated product $\Hig{q}{f}=G_{1,\ldots,f-1}\ast_{F_{1,f-1}}G_{f-1,f,1}$, a property similar to \ref{prop:britton_G_1_f} holds:
\begin{proposition}\label{prop:britton_H}
Let $w\streq w_1w_2\ldots w_s$ be a non-empty sequence with $w_i\in\{(u,k)_i\>:\>1\le i<f-1\}^\ast$ or $w_i\in\{(u,k)_i\>:\>f-1\le i\le f\}^\ast$, alternatingly. If $w$ equals $1$ in $\Hig{q}{f}$, then $w_i\in F_{1,f-1}$ for some index $i$. \qed
\end{proposition}

\begin{algorithm}\label{alg:wp_Hig}
\Input{a word $w$ over $a_i,a_i^{-1}$ ($1\le i\le f$)}
\Output{the answer to $w\overset{?}{=}1$ in $\Hig{q}{f}$}
\BlankLine
Rewrite the input $w$ by replacing each $a_i^{\pm 1}$ by $(\pm 1,0)_i$.\;\label{alg:wp_Hig:init}
Break $w$ into subsequences $w=w_1w_2\ldots w_s$ such that in each $w_j$ the subscripts of all pairs are either in $\{1,2,\ldots,f-2\}$ or in $\{f-1,f\}$.\;
Let $t:=0$.\;
\While{$(t=0\wedge s>1)\vee(0<t<s)$}
{
	\eIf{$t=0\wedge s>1$}
	{
		$\mathcal{L}^\prime$-reduce $w_1$. If $w_1$ becomes emtpy, remove it (thereby decreasing $s$) and continue with the next iteration.\;
		\eIf{$w_1\in F_{1,f-1}$}
		{
			Merge $w_1$ and $w_2$. Before doing so, if $w_1$ and $w_2$ are from different groups (one from $G_{1,\ldots,f-1}$ and the other one from $G_{f-1,f,1}$), swap all the pairs in $w_1$ using the following rules: $(x,0)_1\leftrightarrow(0,x)_f$ and $(0,x)_{f-2}\leftrightarrow(x,0)_{f-1}$\;
		}
		{Increment $t$ by one.\;}
	}({\quad($0<t<s$)})
	{
		\eIf{$w_t$ and $w_{t+1}$ are both from $G_{1,\ldots,f-1}$ or both from $G_{f-1,f,1}$}
		{
			Merge $w_t$ and $w_{t+1}$.\;
			Decrement $t$ by one.\;
		}
		{
			$\mathcal{L}^\prime$-reduce $w_{t+1}$. If $w_{t+1}$ becomes emtpy, remove it and continue with the next iteration.\;
			\eIf{$w_{t+1}\in F_{1,f-1}$}
			{
				Perform the replacements $(x,0)_1\leftrightarrow(0,x)_f$ and $(0,x)_{f-2}\leftrightarrow(x,0)_{f-1}$ in $w_{t+1}$, then merge it with $w_t$.\;
				Decrement $t$ by one.\;
			}
			{
				Increment $t$ by one.\;
			}
		}
	}
}
\Return whether $s=0$.\;
\caption{Procedure for solving the word problem in $\Hig{q}{f}$}
\end{algorithm}

From this proposition, we can derive Algorithm \ref{alg:wp_Hig} which solves the word problem in $\Hig{q}{f}$. In this algorithm, the word $w=w_1w_2\ldots w_s$ is split into $w_1\ldots w_t$ and $w_{t+1}\ldots w_s$. The first part is an $\mathcal{L}^\prime$-reduced alternating sequence of group elements from $G_{1,\ldots,f-1}$ or $G_{f-1,f,1}$ with no $w_i$ ($1\le i\le t$) being in the subgroup $F_{1,f-1}$. In each loop cycle either $t$ increases, or $t$ decreases by one and at the same time, some $w_i$ is merged with $w_{i+1}$, which means that $s$ decreases. Thus, the loop is executed only linearly often. 

In order to get a time bound for Algorithm \ref{alg:wp_Hig}, it remains to show how to perform the tests arithmetic operations on the pairs $(u,k)_i$ efficiently. 

Since power circuits are designed to work with integers, we have to avoid fractions for the first components of pairs $(u,k)_i\in G_{i,i+1}\simeq\Z[1/q]\rtimes\Z$. Therefore, we use the triple notation introduced in \cite{dlu12efficient}. For $u,x,k\in\Z$ with $x\le 0\le k$, let
\begin{equation*}
	[u,x,k]_i:=(u\cdot q^x,x+k)_i\in G_{i,i+1}.
\end{equation*}

If $U$, $X$, and $K$ are markings in a base $q$ power circuit, we call $T_i=[U,X,K]_i$ a triple marking\index{triple marking} and define its value by $\e(T_i)=[\e(U),\e(X),\e(K)]_i\in G_{i,i+1}$. 

Any element of $G_{i,i+1}$ can be written as a triple, but not in a unique way. For instance, $[2,0,0]_i=(2,0)_i=[4,-1,1]_i$, if $q=2$. 
The group operations translate to formulae for multiplication and inversion of triples:
\begin{align*}
	[u,x,k]\cdot[v,y,\ell]&=[u\cdot q^{-y}+v\cdot q^k,x+y,k+\ell]\cr
	[u,x,k]^{-1}&=[-u,-k,-x]
\end{align*}
Furthermore, $[u,x,k]_i\in\g{a_i}\le G_{i,i+1}$ if and only if $x=-k$ and $u\cdot q^x\in\Z$. Similarly, $[u,x,k]_i\in\g{a_{i+1}}\le G_{i,i+1}$ if and only if $u=0$, and finally $[u,x,k]_i$ is the group identity if and only if $u=0$ and $x=-k$. 

At the beginning of Algorithm \ref{alg:wp_Hig}, we create a power circuit with base $q$ consisting of a single node $u$ with $\e(u)=1$. We represent each pair $(u,k)_i$ by a triple marking. Of the three markings in each initial triple, two are zero (empty) and the third has either value $+1$ or $-1$ and can be created using $u$. Let $\omega$ be the sum of the sizes of (the supports of) all these markings. We call $\omega$ the weight of the circuit. From the multiplication formula for triples we see that $\omega$ never increases during the algorithm, keeping in mind that after an operation we can ``forget'' the operand and just keep the result. The initial value of $\omega$ is exactly $n=\abs{w}$. 

After step \ref{alg:wp_Hig:init}, we reduce the circuit, which takes $\O(n^2)$ time. From now on, we keep $\Pi$ reduced following the strategy proposed in Remark \ref{rem:pc_strategy}. 
The swapping operation $(x,0)\leftrightarrow(0,x)$ works in the following way for triples: 
\begin{align*}
	[u,x,k]&\mapsto
	\begin{cases}
		[0,0,u\cdot q^x]&\text{if $x=-k$ and $q^x\mid u$ and $u\ge 0$}\cr
		[0,u\cdot q^x,0]&\text{if $x=-k$ and $q^x\mid u$ and $u<0$}
	\end{cases}\cr
	[0,x,k]&\mapsto[x+k,0,0]
\end{align*}

The whole algorithm computes $\O(n)$ many times an $\mathcal{L}^\prime$-irreducible word. Each of these computations necessitates $\O(n)$ arithmetic operations (and subsequent calls to \ERed). The circuit size remains bounded by $\O(n^2\cdot\omega)\subseteq\O(n^3)$. Thus, one call of \ERed takes $\O(n^3\cdot\omega)\subseteq\O(n^4)$ time. We get a total time bound of $\O(n^6)$. 

This concludes the proof of Theorem \ref{thm:wp_Hig}. 

\section{Conclusion}\label{sec:conclusion}

We have shown that the word problem for the generalized Higman groups $\Hig{q}{f}$ is solvable in polynomial time. An important ingredient for this result was the extension of the power circuit data structure to arbitrary bases $q\ge 2$. From an algorithmic point of view, this is an interesting result in itself and may provide a useful tool in group theory as well as other areas. 

The techniques used in this paper do not apply to the even more general groups $H_f(p,q)$ and $G_{(p,q)}$, where the underlying Baumslag-Solitar group $\BS(1,q)$ is replaced by $\BS(p,q)=\gr{a,t}{ta^pt^{-1}=a^q}$ for some $p\ge 1$. This is because $\BS(p,q)$ is not a semi-direct product when $p>1$. The word problem for these groups is open. Note that even for $f<4$ the group $H_f(p,q)$ can be non-trivial if $p>1$. 

\addcontentsline{toc}{section}{Bibliography}

\end{document}